\documentclass{amsart}
\usepackage{amsfonts,amssymb,euscript,mathrsfs,color}
\usepackage[all]{xy}

\usepackage{hyperref}

\numberwithin{equation}{section}

\newtheorem{theorem}{Theorem}[section]

\newtheorem{lemma}[theorem]{Lemma}
\newtheorem{proposition}[theorem]{Proposition}

\theoremstyle{definition}

\theoremstyle{remark}

\providecommand{\abs}[1]{\left\vert#1\right\vert}
\providecommand{\babs}[1]{\big\vert#1\big\vert}

\providecommand{\nm}[1]{\left\Vert#1\right\Vert}

\providecommand{\bnm}[1]{\big\Vert#1\big\Vert}

\def\dt{\partial_t}
\def\p{\partial}
\def\ls{\lesssim}
\def\rt{\rightarrow}
\def\r{\mathbb{R}}
\def\no{\nonumber}

\def\ui{\mathrm{i}}
\newcommand{\eps}{\varepsilon}

\def\R{\mathbb{R}}
\def\C{\mathbb{C}}
\def\Z{\mathbb{Z}}

\def\z{\mathbb{Z}}
\def\s{\mathscr{S}}

\def\D{\Delta}

\def\tr{\text{tr}}
\def\j{\mathfrak{I}}

\def\ah{\alpha}
\def\bh{\beta}

\def\D{\Delta}
\def\DD{\Delta_{{\rm d}}}

\def\ph{\psi^h}

\def\ep{\varepsilon}

\def\l{\ell}

\def\pe{\psi}

\newcommand{\qtq}[1]{\quad\text{#1}\quad}
\let\Re=\undefined
\DeclareMathOperator{\Re}{Re}
\let\Im=\undefined
\DeclareMathOperator{\Im}{Im}
\DeclareMathOperator{\supp}{supp}
\DeclareMathOperator{\sign}{sign}

\allowdisplaybreaks

\begin{document}

\title{Continuum Limit for the Ablowitz--Ladik System}

\author{Rowan Killip}
\address{Department of Mathematics, University of California, Los Angeles, CA 90095, USA}
\email{killip@math.ucla.edu}

\author{Zhimeng Ouyang}
\address{Institute for Pure and Applied Mathematics, University of California, Los Angeles, CA 90095, USA}
\email{zhimeng\_ouyang@alumni.brown.edu}

\author{Monica Visan}
\address{Department of Mathematics, University of California, Los Angeles, CA 90095, USA}
\email{visan@math.ucla.edu}

\author{Lei Wu}
\address{Department of Mathematics, Lehigh University}
\email{lew218@lehigh.edu}

\begin{abstract}
We show that solutions to the Ablowitz--Ladik system converge to solutions of the cubic nonlinear Schr\"odinger equation for merely $L^2$ initial data.  Furthermore, we consider initial data for this lattice model that excites Fourier modes near both critical points of the discrete dispersion relation and demonstrate convergence to a \emph{decoupled} system of nonlinear Schr\"odinger equations.   
\end{abstract}

\maketitle

\section{Introduction}

The Ablowitz--Ladik system, introduced in \cite{MR0377223}, describes the evolution of a field $\alpha:\Z\to\C$.  It comes in two variants, focusing and defocusing.  Both may be written as
\begin{align}\label{AL}\tag{AL}
\ui\dt\alpha_n &= -\big(\alpha_{n-1}-2\alpha_n+\alpha_{n+1}\big)
+\alpha_n\beta_n\big(\alpha_{n-1}+\alpha_{n+1}\big)
\end{align}
by adopting the expedient that $\beta_n:=\overline\alpha_n$ in the defocusing case and $\beta_n:=-\overline\alpha_n$ in the focusing case.  (This convention will remain in force throughout the paper.)

Ablowitz and Ladik introduced this model as a discrete form of the one-dimensional cubic Schr\"odinger equation,
\begin{align}\label{NLS}\tag{NLS}
\ui\partial_t \psi = - \Delta \psi \pm 2 |\psi|^2\psi,
\end{align}
that preserves its complete integrability.

Given this connection, it is quite natural to imagine that solutions of \eqref{AL} provide an accurate means of simulating solutions of \eqref{NLS}, at least for slowly-varying (high-regularity) initial data.  Indeed, this intuition is backed up by a number of studies that we will discuss below.  The question at the heart of this paper, however, is this: to what extent do solutions to \eqref{AL} and \eqref{NLS} parallel one another for low-regularity initial data?

We will affirm the intuition laid out above still more strongly than previous authors by showing convergence for merely $L^2$ initial data.  This constitutes a significant expansion of the class of initial data relative to previous investigations such as \cite{MR3939333,MR3009717} which require the initial data to lie in $H^1$. We achieve this through the introduction of a new method that synthesizes compactness and Strichartz-based techniques.

As a counter-point to our convergence result, we will also demonstrate a certain naivet\'e in the logic enunciated earlier by studying solutions to \eqref{AL} with initial data
\begin{align}\label{dumb double data}
\alpha_n(0) = h \psi_0 (hn) +  (-1)^n h \phi_0 (hn),
\end{align}
which combines slowly varying and rapidly oscillating initial data.  Here, $h$ is the length scale associated to the continuum approximation; we are studying the $h\to0$ limit.  Note the prefactor $h$ appearing in \eqref{dumb double data}; this ensures the balance between the dispersion and nonlinear effects.  It is solely in this regime that one expects a nonlinear dispersive limit as $h\to 0$. In this regime, the natural time scale is also stretched --- solutions of \eqref{NLS} over the time interval $[-T,T]$ correspond to those of \eqref{AL} over the longer period $[-h^{-2}T,h^{-2}T]$.

We will show that initial data of the form \eqref{dumb double data} leads to solutions to a system of nonlinear Schr\"odinger equations:
\begin{align}\label{NLS system}
    \ui\dt\pe=-\D\pe\pm2\abs{\pe}^2\pe \qtq{and}
    -\ui\dt\phi=-\D\phi\pm2\abs{\phi}^2\phi,
\end{align}
in which --- and this was a surprise to us --- the two frequency components evolve independently!  As we will see, this absence of interactions is only observed asymptotically and will be traced to a certain nonresonance in spacetime.  In our analysis of initial data of the form \eqref{dumb double data}, we will allow data $\phi_0,\psi_0$ that is merely $L^2$, thereby unifying the two main themes of this paper.  

An important phenomenological parallel to our discovery of \eqref{NLS system} in the continuum limit of \eqref{AL} has appeared previously in the link between KdV and the Toda lattice.  In \cite{MR3366652,MR3327553} it was demonstrated that the Toda lattice must be modeled by a \emph{pair} of KdV equations.  We also direct the reader to \cite{MR1870156} which considers more general (non-integrable) lattice models.

Complete integrability of the Toda lattice is at the very heart of the methods of \cite{MR3366652,MR3327553}, which focus on the description of action-angle variables.  This is very different from what we shall be doing; indeed, much of our analysis is based directly on Strichartz estimates.  The complete integrability of \eqref{AL} does play a small role in our arguments, namely, in demonstrating equicontinuity.  Nevertheless, it is reasonable to imagine that our approach could be expanded to cover non-integrable lattice approximations to \eqref{NLS}, obtaining equicontinuity via almost conservation laws and Strichartz estimates.  We do not pursue this here because we are specifically interested in the Ablowitz--Ladik system due its link to spin chain dynamics.

This link was first discovered in \cite{Ishimori} and is a discrete analogue of the famous Hasimoto transform linking the continuum Heisenberg model to \eqref{NLS}. Indeed, our desire to understand the continuum limit of low-regularity solutions to \eqref{AL} is fueled by the long-term goal, elaborated in \cite{MR4049393}, of constructing dynamics for the continuum Heisenberg model in its Gibbs state.

Via the discrete Hasimoto transform, the analysis of the spin model in the Gibbs state is converted to the study of \eqref{AL} with $\alpha_n(t)$ being certain specific (non-Gaussian) i.i.d. random variables.  Such initial data excites Fourier modes throughout the circle to an equal degree and consequently both terms in \eqref{dumb double data} may be considered equally significant. In \cite{MR4049393}, the first and third author together with Angelopoulos constructed global dynamics to \eqref{AL} for this type of data, and proved invariance of this white noise measure under the \eqref{AL} flow.  Moreover, the discrete spin chain model was shown to admit global solutions in the Gibbs state and these were shown to preserve  the Gibbs measure.

In order to rigorously formulate our results, we need to explain how we pass between functions on the line $\R$ and the lattice $\Z$.  Clearly, the sampling formula \eqref{dumb double data} does not make sense in the $L^2$ setting we will be studying; indeed, point evaluation is not continuous on $H^s(\R)$ unless $s>\frac12$.  The remedy is to perform a mild smoothing of the continuum initial data before sampling:
\begin{align}\label{E:initial data}
\alpha_n (0) = h [P_{\leq N} \psi_0] (hn) +  (-1)^n h [P_{\leq N} \phi_0] (hn).
\end{align}
Here $P_{\leq N}$ denotes the sharp projection to frequencies $|\xi|\leq N$.  For concreteness, we will relate $h$ and $N$ by a power law:
\begin{align}\label{N from h}
N = h^{-\gamma} \qtq{with} 0< \gamma \leq \tfrac{13}{18} .
\end{align}
As $h\to 0$, so $N\to\infty$, thus revealing the full irregularity of the initial data; cf.~\eqref{all}.

The choice of sharp cutoffs is made for no other reason than expository simplicity.  Such multipliers are much less local and are unbounded on $\ell^\infty$.  Our use of such cutoffs is indicative of the robustness of our arguments and also saves the reader from considering a multitude of bump functions concomitant with traditional Littlewood--Paley theory.

While a signal-processing philosophy would suggest choosing $N$ as the Nyquist frequency (and no higher), dynamical considerations make $N\approx h^{-1}$ ill-suited to the problem at hand.  The principal obstruction appears at the linear level: The dispersion relation for the discrete Schr\"odinger equation is $\omega=2-2\cos(\theta)$.  This is only approximately quadratic near $\theta=0$ and $\theta=\pi$ (modulo $2\pi$), which are the only regimes where we may expect Schr\"odinger--like behavior.  Near the inflection points $\theta=\pm\pi/2$, the discrete model has very weak dispersion.  This in turn leads to a wholly different mKdV-like dynamics with a different characteristic time scale.  Initial data with significant excitation near the inflection points will be treated in future work.

As the \eqref{AL} dynamics are nonlinear, the Fourier support of a solution is not preserved in time; one fully expects excitations to spread from the initial data to neighborhoods of the inflection points.  Nevertheless, we are able to control the extent of such transfer and so demonstrate that the choice \eqref{N from h} suffices to suppress these unwanted mKdV dynamics; see Proposition~\ref{P:suppression}.

Based on our analysis, which we believe gives an accurate estimate of the transfer of $\ell^2$ norm to neighborhoods of the inflection points, we would argue that should one be employing \eqref{AL} as a numerical scheme for the simulation of NLS (an idea advocated for in \cite{MR1308108}, for example), then one should first band-limit the initial data in the manner indicated by \eqref{N from h} so as to avoid spurious results generated by the deficiencies of the discrete dispersion relation.

For readers with a particular interest in numerical schemes for NLS (rather than purely \eqref{AL}), we note that the inflection-point issue can also be avoided by combining Fourier truncation of the initial data in concert with a different discrete model.   Two such approaches are discussed in \cite{MR2485456,MR2980459}: One may perform Fourier truncations on the nonlinearity (which preserves the Hamiltonian structure) or one may introduce a mild form of viscosity at the inflection points to suppress such unwanted excitations.

To compare the solution of \eqref{AL} with initial data \eqref{E:initial data} to that of \eqref{NLS system}, we must (at each moment of time) generate \emph{two} functions, $\psi(t,x)$ and $\phi(t,x)$, on the real line from the single lattice function $\alpha_n(h^{-2}t)$.  To do this, we simply split the Fourier transform $\widehat\alpha$ into two pieces using a sharp cutoff to either semicircle:
\begin{gather}\label{psi phi def}
\begin{aligned}
\widehat{\psi^h}(t,\xi)&= \widehat{\alpha}(h^{-2}t, h\xi) \chi_{(-\frac{\pi}2,\frac\pi 2)}(h\xi),\\
\widehat{\phi^h}(t,\xi)&= e^{4\ui h^{-2}t} \widehat{\alpha}(h^{-2}t, h\xi+ \pi ) \chi_{(-\frac{\pi}2,\frac\pi 2)}(h\xi).
\end{aligned}
\end{gather}
or equivalently,
\begin{gather}\label{psi phi def'}
\begin{aligned}
\psi^h(t,x) &= h^{-1} \int_{|\theta|<\frac{\pi}2} e^{\ui \frac{x}{h}\theta}\,\widehat{\alpha}(h^{-2}t,\theta) \tfrac{d\theta}{2\pi},\\
\phi^h(t,x) &= h^{-1} e^{4\ui h^{-2}t} \int_{|\theta-\pi|<\frac{\pi}2} e^{\ui \frac{x}{h}(\theta-\pi)}\,\widehat{\alpha}(h^{-2}t,\theta) \tfrac{d\theta}{2\pi}.
\end{aligned}
\end{gather}

One small quibble remains before we may present our main result: Do \eqref{AL} and \eqref{NLS system} actually admit global solutions?  In the case of \eqref{NLS system}, the answer is unequivocally yes: Tsutsumi \cite{MR915266} showed that these equations are globally well-posed on $L^2(\R)$. Local well-posedness is shown by contraction mapping in Strichartz spaces; this is then rendered global using the conservation of the $L^2$ norm.

Regarding \eqref{AL}, we see that local well-posedness in $\ell^2(\Z)$ is trivial: RHS\eqref{AL} is locally Lipschitz! In Section~\ref{S:3} we will show how conservation laws guarantee that such solutions exist globally, at least for
\begin{align}\label{small h}
h\leq h_0 :=\min\Bigl\{1,\tfrac1{100}\bigl[\|\psi_0\|_{L^2}^2 + \|\phi_0\|_{L^2}^2 \bigr]^{-1}\Bigr\}.
\end{align}
As we wish to send $h\to 0$ this hypothesis is of no real consequence; it exists solely to address a singularity in the natural conservation laws (and also the symplectic structure) in the defocusing case.  Indeed, in the defocusing case, it is natural to regard the phase space as comprised only of maps $\alpha:\Z\to\mathbb{D}$.

The main result of this paper is the following, which we present schematically in Figure~\ref{F:1}: 

\begin{theorem}\label{T:main}
Fix $\psi_0,\phi_0\in L^2(\R)$ and let $\psi,\phi\in (C_tL^2_x\cap L_{t,loc}^6L_x^6)(\R\times\R) $ denote the unique global solutions of \eqref{NLS system} with this initial data.

Given $h>0$ satisfying \eqref{small h}, let $\alpha_n(t)$ be the global solution to \eqref{AL} with initial data specified by \eqref{E:initial data} and \eqref{N from h} and let $\psi^h,\phi^h:\R\times\R\to \C$ be the corresponding spacetime representatives of this solution built via \eqref{psi phi def'}. Then 
\begin{align}\label{E:T:main}
    \psi^h(t,x)\rt\psi(t,x) \qtq{and} \phi^h(t,x)\rt\phi(t,x)
\end{align}
in $C_tL^2_x([-T,T]\times\r)$ for any $T>0$.
\end{theorem}

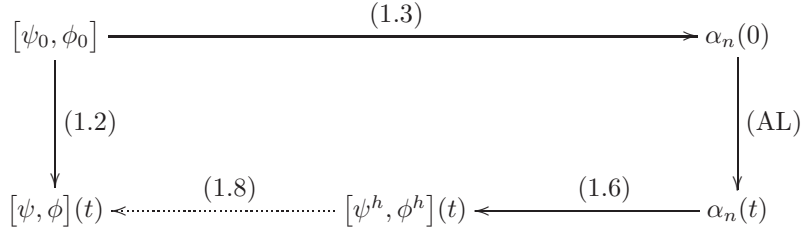
\begin{figure}
\begin{center}\mbox{
\xymatrix@C=7pc@R=4pc{
\big[\psi_0,\phi_0\big] \ar[rr]^{\txt{\eqref{E:initial data}}} \ar[d]^{\txt{\eqref{NLS system}}}
	& & \alpha_{n}(0) \ar[d]^{\txt{\eqref{AL}}} \\
\big[\psi,\phi\big](t) & 
\big[\psi^h,\phi^h\big](t) \ar@{.>}[l]_(.5){\txt{\eqref{E:T:main}}}
	&  \alpha_n(t) \ar[l]_(0.4){\txt{\eqref{psi phi def'}}}
}}\end{center}
\caption{A schematic representation of Theorem~\ref{T:main}.}\label{F:1}
\end{figure}

\subsection{Outline of the proof} Broadly speaking, our argument is one of compactness/uniqueness: we show that for every sequence of parameters $h\to0$ there is a convergent subsequence of discrete solutions (transferred to the line via \eqref{psi phi def'}).  We then show that all such subsequential limits are identical because they obey the same integral equations and those integral equations have unique solutions.

The integral equations we refer to here are nothing more than the Duhamel formulas for solutions of \eqref{NLS system}; see \eqref{duhamel 1} and \eqref{duhamel 2}.  Naturally, these equations contain the cubic nonlinearity.  This is a problem for initial data that is merely in $L^2$ --- cubing an $L^2$ function need not yield a distribution.   If we cannot even determine whether or not a $C_t L^2_x$ function is a solution of \eqref{NLS}, uniqueness becomes hopeless!

The remedy is provided by Strichartz estimates.  As shown in \cite{MR915266}, solutions with $L^2(\R)$ initial data can be constructed via contraction mapping in $L^6_{t,x}([-T,T]\times\R)$; these solutions are unique in this class and the cubic nonlinearity makes sense as a spacetime distribution. To exploit this form of uniqueness, we must prove $L^6_{t,x}([-T,T]\times\R)$ bounds for the functions $\psi^h,\phi^h$.

As mentioned earlier and thoroughly discussed in \cite{MR3939333}, the presence of inflection points in the discrete dispersion relation presents a major obstacle to controlling the discrete solutions in Strichartz spaces.  Our new remedy to this problem is to begin our analysis by controlling the dynamical redistribution of $\ell^2$ norm among frequencies; see Proposition~\ref{P:suppression}.  In Proposition~\ref{Prop:strichartz-bounds}, we show that the frequency control we achieve is strong enough to recover from the inevitable losses in the discrete Strichartz estimates originating from the inflection points.

Let us turn now to the question of compactness, which is addressed in Section~\ref{sec:precompactness}; see Theorem~\ref{T:precompactness}.  As we review at the beginning of that section, precompactness rests on three pillars: boundedness, equicontinuity, and tightness.  Boundedness is deduced easily from the conservation laws of \eqref{AL}, specifically, from the analogue of $L^2$ conservation.

There are actually two equicontinuity requirements because we must allow translations both in space and in time.  Equicontinuity in space can also be formulated as tightness of the Fourier transform; see \eqref{equi 2}.  In this guise, it is easier to see that it is amenable to attack via conservation laws and indeed, this is how we prove it.

Both equicontinuity in time and tightness rely on the high-frequency control provided by equicontinuity in space.  Even under the linear flow, high-frequency wave packets travel very fast, which is antithetical to both tightness and equicontinuity in time.   Of the two, tightness is the more delicate to prove because the natural microscopic conservation law does not interact well with frequency cutoffs.  This is a peculiarity of \eqref{AL} not present in \eqref{NLS} or other discrete analogues.

Section~\ref{Sec:Convergence} is primarily devoted to showing that any subsequential limits of discrete solutions actually solve the Duhamel integral equation.  This relies on all of the preceding.  Strichartz control is essential to overcome the cubic power in the nonlinearity, for example, while equicontinuity in space is needed to replace the nonlinearity in \eqref{AL}, which involves values at adjacent lattice points, with a purely on-site nonlinearity.  Perhaps the most surprising element of this analysis is the fact that the two frequency components $\psi^h$ and $\phi^h$ of the solution do not interact.  The explanation for this has both algebraic and analytic components.  The fact that the \eqref{AL} nonlinearity involves adjacent sites produces a key sign change in \eqref{sign flip} which cancels the naive interactions $|\psi|^2\phi$ and $|\phi|^2\psi$.  However, two prima face significant interactions remain. We prove that these drop out of the $h\to0$ limit due to a nonresonance phenomenon --- they oscillate in time at a frequency far removed from that of the linear solutions; see Lemma~\ref{L:3}.

At the end of Section~\ref{Sec:Convergence}, we close the paper with a quick review of how these results complete the proof of Theorem~\ref{T:main}.

\subsection*{Acknowledgements}

Rowan Killip is supported by NSF grants DMS-1856755 and DMS-2154022. 
Monica Visan is supported by NSF grants DMS-1763074 and DMS-2054194. 
Lei Wu is supported by NSF grant DMS-2104775.

\section{Preliminaries}\label{S:2}

Throughout this paper, $C$ will denote a constant that does not depend on the initial data or on $h$, and which may vary from one line to another.
We write $A \lesssim B$ or $B\gtrsim A$ whenever $A\leq CB$ for some constant $C>0$.  We write $A\simeq B$ whenever $A\lesssim B$ and $B\lesssim A$.  

We use $L_t^q L_x^r$ to denote the spacetime norm
$$
\|F\|_{L_t^q L_x^r(\R\times\R)} :=\Bigl(\int_{\R}\Bigl(\int_{\R} |F(t,x)|^r \,dx \Bigr)^{q/r} \,dt\Bigr)^{1/q},
$$
with the usual modifications when $q$ or $r$ is infinity, or when the domain $\R \times \R$ is replaced by some smaller spacetime region. When $q=r$ we abbreviate $L_t^q L_x^r$ by $L^q_{t,x}$.  The discrete version of the spacetime norm $L_t^q L_x^r(\R\times\R)$ is $L_t^q \ell_n^r(\R\times\Z)$.

The Hilbert--Schmidt norm of a bounded linear operator $\Omega: \l_n^2(\z)\rt \l_n^2(\z)$ is the $\ell^2$ norm of its matrix entries:
\begin{align*}
    \nm{\Omega}_{\j_2}^2:=\sum_{n,m\in\z}\bigl|\Omega_{nm}\bigr|^2.
\end{align*}
By the H\"older inequality, if $\Omega_1,\Omega_2,\cdots,\Omega_m$ with $m\geq 2$ are Hilbert--Schmidt operators, then their product is trace class and 
\begin{align*}
\babs{\tr\big\{\Omega_1\Omega_2\cdots\Omega_m\big\}}\leq \prod_{i=1}^m\nm{\Omega_i}_{\j_2}.
\end{align*}

Our convention for the Fourier transform on the line will be
\begin{align*}
    \widehat{f}(\xi) =\int_{\r}f(x) e^{-\ui x\xi}\,dx \quad\text{so that}\quad f(x)=\int_{\r}\widehat{f}(\xi) e^{\ui x\xi}\,\tfrac{d\xi}{2\pi},
\end{align*}
while in the discrete case we will use
$$
\widehat f(\theta) = \sum_{n\in \Z} f(n) e^{-\ui n\theta} \quad\text{so that}\quad f(n) = \int_{-\pi}^{\pi} \widehat f(\theta) e^{\ui n\theta}\, \tfrac{d\theta}{2\pi}.
$$
These definitions of the Fourier transform yield the Plancherel identities
\begin{align*}
    \int_{\r}\babs{f(x)}^2\, dx = \int_{\r}\babs{\widehat{f}(\xi)}^2\,\tfrac{d\xi}{2\pi} \quad\text{and}\quad \sum_{n\in \Z} |f(n)|^2= \int_{-\pi}^{\pi} |\widehat f(\theta)|^2\,\tfrac{d\theta}{2\pi}.
\end{align*}

With these conventions, the Poisson summation formula takes the form 
\begin{align*}
    \sum_{n\in\z}hf(nh)e^{-\ui n\theta}=\sum_{m\in\z}\widehat{f}\left(\tfrac{\theta+2\pi m}{h}\right) \qtq{for any} f\in\s(\R).
\end{align*}
Therefore, if $\supp(\widehat{f}\,)\subset\bigl[-\frac{\pi}{h},\frac{\pi}{h}\bigr]$ and $|\theta|\leq\pi$, then 
\begin{align*}
    \sum_{n\in\z}hf(nh)e^{-\ui n\theta}=\widehat{f}\left(\tfrac{\theta}{h}\right).
\end{align*}
Consequently, for the initial data \eqref{E:initial data}, we have
\begin{align}\label{data supp}
    \widehat{\ah}(0,\theta)=\sum_{n\in\z}\ah(0,n)e^{-\ui n\theta}=\widehat{P_{\leq N}\psi_0}\left(\tfrac{\theta}{h}\right)+\widehat{P_{\leq N}\phi_0}\left(\tfrac{\theta-\pi}{h}\right)
\end{align}
for $\theta\in[-\tfrac\pi2,\tfrac{3\pi}2]$ and $2Nh<\pi$.  In particular, recalling that $N=h^{-\gamma}$, we have 
\begin{align}\label{all}
\lim_{h\to 0} h^{-1}\|\alpha(0)\|_{\ell^2}^2 = \lim_{h\to 0}\|P_{\leq N} \psi_0\|_{L^2}^2 + \|P_{\leq N}\phi_0\|_{L^2}^2 =  \|\psi_0\|_{L^2}^2 + \|\phi_0\|_{L^2}^2.
\end{align}

\begin{lemma}\label{L:sums to int}
If $f,g \in L^2(\R)$ satisfy $\supp(\widehat f\,)\subseteq\bigl[-\frac\pi h,\frac \pi h\bigr]$ and $\supp(\widehat g)\subseteq\bigl[-\frac\pi h,\frac \pi h\bigr]$, then
\begin{align*}
\int f(x) \,\overline{g(x)} \, dx= h\sum_n f(nh) \,\overline{g(nh)}.
\end{align*}
\end{lemma}

\begin{proof}
Expanding $\widehat f$ and $\widehat g$ with respect to the orthonormal basis $\{h^{\frac12} e^{-\ui nh\xi}:\, n\in \Z\}$ of $L^2([-\frac\pi h,\frac \pi h]; \frac{d\xi}{2\pi})$, we obtain
\begin{align*}
\int f(x) \,\overline{g(x)} \, dx= \int_{-\frac\pi h}^{\frac\pi h} \widehat f(\xi) \,\overline{\widehat g(\xi)} \,\tfrac{d\xi}{2\pi}
&=\sum_n \int h^{\frac12} e^{\ui nh\xi}\widehat f(\xi)\, d\xi \, \overline{\int h^{\frac12} e^{\ui nh\eta}\widehat g(\eta)\, d\eta}\\
&= h\sum_n f(nh)\, \overline{g(nh)}. \qedhere
\end{align*}
\end{proof}

It is not difficult to upgrade the previous argument to show that $f\mapsto \sqrt{h} f(nh)$ is actually a unitary map --- a result known as Shannon's sampling theorem.  For what follows, we will also need to understand the $L_x^p \to \ell_n^p$ properties of this mapping (on band limited functions).  This was addressed already by Plancherel and Polya in \cite{PP}:

\begin{lemma} \label{Lem:norm-relation}
Fix $1<p<\infty$ and let $f:\R\to \C$ be a Schwartz function with $\supp(\widehat{f}\,)\subseteq\bigl[-\frac\pi h,\frac \pi h\bigr]$.  Then
\begin{align} \label{lp-Lp}
    h\sum_{n\in\z}|f(nh)|^p\simeq \int_{\r}|f(x)|^p\, dx.
\end{align}
\end{lemma}

It is well-known that operators defined by sharp (as opposed to smooth) Fourier cutoffs are $L^p$-bounded for $1<p<\infty$; indeed, this follows immediately from the well-known $L^p$-boundedness of the Hilbert transform,
$$
\widehat{Hf} (\xi) = -\ui\sign(\xi) \widehat{f}(\xi) \qtq{or equivalently,} [Hf](x) = \text{P.V.}\int \frac{f(y)}{x-y}\,\frac{dy}{\pi} ,
$$
which was proved by M.~Riesz \cite{Riesz1928}.  In Section~23 of this very same paper, Riesz considered the discrete Hilbert transform and showed that it is $\ell^p$-bounded for all $1<p<\infty$.  It follows that for any arc $[a,b]$ on the circle, the mapping
$\widehat\alpha(\theta) \mapsto \chi_{[a,b]}(\theta) \widehat\alpha(\theta)$
is bounded on $\ell^p(\Z)$.

Our next two results concern the linear operator arising in \eqref{psi phi def'}, which transfers solutions of \eqref{AL} to functions on $\R$.  Specifically, we mean the mapping
\begin{align}\label{R def}
[\mathcal R c] (x) := h^{-1} \int_{-\frac\pi2}^{\frac\pi2} e^{\ui\frac{x}{h}\theta} \widehat c(\theta)\tfrac{d\theta}{2\pi} = \text{P.V.}\sum_n \frac{\sin(\tfrac{\pi}{2h}[x-nh])}{\pi(x-nh)} c_n
\end{align}
which we will show maps $\ell^2_n$ to $L^2_x$.  Note that \eqref{psi phi def'} can be rewritten as
$$\psi^h(t) = \mathcal R [\alpha_n(h^{-2}t)] \qtq{and}\phi^h(t) = e^{4\ui h^{-2}t} \mathcal R [(-1)^n\alpha_n(h^{-2}t)].$$

\begin{lemma}\label{L:C}
The operator $\mathcal R$ is a bounded operator from $\ell_n^2(\Z)$ to $L^2_x(\R)$ with norm
$$
\|\mathcal R\|_{\ell_n^2(\Z)\to L^2_x(\R)}\lesssim h^{-\frac12}.
$$
\end{lemma}

\begin{proof}
The result is an immediate consequence of Lemma~\ref{L:sums to int}.  Indeed,
\begin{equation*}
\|\mathcal R c\|_{L_x^2}= h^{\frac12} \|[\mathcal R c](nh)\|_{\ell^2_n} = h^{\frac12} \|\widehat c\|_{L^2([-\frac\pi 2, \frac\pi 2])}\leq h^{\frac12} \|c\|_{\ell^2_n}.\qedhere
\end{equation*}
\end{proof}

\begin{lemma}\label{L:tighty} Let $c \in \ell^2_n$ be a sequence supported in the interval $[-h^{-1}L,h^{-1}L]$. Then for any $L'>0$ we have
\begin{align}\label{E:tighty}
\int_{|x|>L+L'} \bigl|[\mathcal R c](x)\bigr|^2\,dx \lesssim \frac{L}{h L'} \| c_n \|_{\ell^2_n}^2 .
\end{align}
\end{lemma}

\begin{proof}
The result follows immediately from two simple observations:
\begin{equation*}
\int_{|x-nh|>L'} \Bigl| \frac{2\sin(\tfrac{\pi}{2h}[x-nh])}{x-nh}\Bigr|^2\,dx \lesssim  \frac{1}{L'}
\qtq{and} \| c \|_{\ell^1_n} \lesssim  h^{-\frac12} L^\frac12 \| c \|_{\ell^2_n}. \qedhere
\end{equation*}
\end{proof}

\section{Conservation Laws}\label{S:3}

As a completely integrable system, \eqref{AL} enjoys infinitely many conservation laws.   Of these, we mention just three:
\begin{gather*}
M(\alpha)=-\sum_{n\in\z}\ln\big(1-\ah_n\bh_n\big), \\
H(\alpha)= - \sum_{n\in\z} \ah_n\bh_{n+1}+\ah_{n+1}\bh_n+2\ln\big(1-\ah_n\bh_n\big) ,\\
H_2(\alpha) =  - \sum_n  2\ln\big(1-\ah_n\bh_n\big) + 2\Re \Big[ \alpha_{n+2}\beta_{n} - \tfrac12 \alpha_n^2\beta_{n-1}^2 - \alpha_{n+1}\alpha_n\beta_n\beta_{n-1} \Bigr] .
\end{gather*}
Note that $H$ also serves as the Hamiltonian for \eqref{AL} with respect to the Poisson structure
\begin{align*}
    \big\{F,G\big\}:=\tfrac{1}{\ui}\sum_{n\in\z}\big(1-\ah_n\bh_n\big)\left(\tfrac{\p F}{\p\ah_n}\tfrac{\p G}{\p\bh_n}-\tfrac{\p F}{\p\bh_n}\tfrac{\p G}{\p\ah_n}\right).
\end{align*}

\begin{proposition}\label{P:mass}
Under the assumption \eqref{small h}, the evolution \eqref{AL} with initial data \eqref{E:initial data} has a global solution in $C_t\ell_n^2$. Moreover,
\begin{align}\label{E:mass bound}
  \bigl| M(\alpha(t))\bigr|  \simeq \nm{\ah(t)}_{\l_n^2}^2\lesssim h\bigl[\|\psi_0\|_{L^2}^2 + \|\phi_0\|_{L^2}^2 \bigr]  \quad\text{uniformly for $t\in \R$}.
\end{align}
\end{proposition}

\begin{proof}
Local well-posedness in $\ell^2$ is guaranteed by Picard's theorem.  Under the assumption \eqref{small h}, the quantity $M(\alpha(t))$ is initially well-defined; moreover, it is conserved by these local solutions.   We will show that such solutions can be extended globally in time by demonstrating the equivalence stated in \eqref{E:mass bound}.  The inequality in \eqref{E:mass bound} then follows from Lemma~\ref{Lem:norm-relation}.

From the power series expansion of $\ln$, we have 
\begin{align}\label{8:52}
\Bigl| \bigl| M(\alpha(t))\bigr|  - \nm{\ah(t)}_{\l_n^2}^2 \Bigr| \leq \sum_{\ell=2}^{\infty}\tfrac{1}{\ell}\nm{\ah(t)}_{\l^2_n}^{2\ell}.
\end{align}
On any time interval where 
\begin{align}\label{8:54}
\nm{\ah(t)}_{\l^2_n}^2\leq \tfrac1{20},
\end{align}
the series on the right-hand side of \eqref{8:52} converges and yields
\begin{align*}
\Bigl| \bigl| M(\alpha(t))\bigr|  - \nm{\ah(t)}_{\l_n^2}^2 \Bigr| \leq \tfrac{1}{20}\nm{\ah(t)}_{\l^2_n}^2
\end{align*}
and consequently,
\begin{align}\label{8:53}
\tfrac12 \bigl| M(\alpha(t))\bigr| \leq \nm{\ah(t)}_{\l_n^2}^2 \leq 2 \bigl| M(\alpha(t))\bigr|.
\end{align}

By Lemma~\ref{Lem:norm-relation}, \eqref{small h} guarantees
\begin{align*}
\|\alpha_0\|_{\ell^2_n}^2 \leq \tfrac1{100}.
\end{align*}
As trajectories are $\ell^2_n$-continuous in time, \eqref{8:53} and the conservation of mass show that \eqref{8:54}, and so also \eqref{8:53}, hold for all time.
\end{proof}

It is elementary to verify that the solutions constructed in Proposition~\ref{P:mass} also conserve $H$ and $H_2$, which are readily seen to be continuous functionals on $\ell^2_n$.
Our next result demonstrates suppression of the influence of the inflection points $\pm \frac\pi 2$ on the solution to \eqref{AL} for all times.

\begin{proposition}\label{P:suppression}
Let $\alpha$ be the solution to \eqref{AL} with initial data \eqref{E:initial data} and $h$ satisfying \eqref{small h}.   For $0<\delta<1$, let $P_\delta$ be the sharp Fourier cutoff defined via
\begin{align}\label{P defn}
\widehat {P_\delta f} (\theta) = \chi_{\mathcal G_\delta}(\theta) \hat f(\theta) \qtq{with} \mathcal G_\delta=\bigl\{\theta\in\R/2\pi\Z : \sin^2(\theta) < \delta^2\bigr\} .
\end{align}
Then
\begin{align}\label{E:suppress}
    \bnm{ [1-P_\delta]\ah(t)}_{\l^{2}_n}\ls \delta^{-1}\big(hN+h^{\frac{1}{2}}\big)\nm{\ah(0)}_{\l^2_{n}}.
\end{align}
\end{proposition}

\begin{proof}
A straightforward computation reveals that the quadratic part of the conserved quantity $H_2(\alpha)$ is given by
\begin{align*}
H_2^{[2]}\bigl(\alpha(t)\bigr) = \pm\int_{-\pi}^\pi 4\sin^2(\theta)  |\widehat{\alpha}(t,\theta)|^2\,\tfrac{d\theta}{2\pi},
\end{align*}
while the higher order terms can be estimated using Proposition~\ref{P:mass} by
\begin{align*}
\Bigl| H_2\bigl(\alpha(t)\bigr) -H_2^{[2]}\bigl(\alpha(t)\bigr)\Bigr|
&\lesssim \|\alpha(t)\|_{\ell_n^2}^4 + \sum_{\ell\geq 2}  \|\alpha(t)\|_{\ell_n^2}^{2\ell} \lesssim \|\alpha(t)\|_{\ell_n^2}^4\lesssim h\|\alpha(0)\|_{\ell_n^2}^2.
\end{align*}
Using the conservation of $H_2$ and recalling \eqref{data supp}, we therefore deduce that
\begin{align*}
\int_{-\pi}^{\pi} 4\sin^2(\theta)|\widehat{\alpha}(t,\theta)|^2\,\tfrac{d\theta}{2\pi}&\lesssim\int_{-\pi}^\pi 4\sin^2(\theta) |\widehat{\alpha}(0,\theta)|^2\,\tfrac{d\theta}{2\pi} +h\|\alpha(0)\|_{\ell_n^2}^2\\
&\lesssim [(hN)^2 + h]\|\alpha(0)\|_{\ell_n^2}^2.
\end{align*}
This completes the proof of the proposition.
\end{proof}

A convenient way of understanding the whole family of conservation laws is through their generating function, which we will discuss next, following the paradigm set forth in \cite{Harrop-Griffiths.Killip.Visan2021}.

For $z\in \C$ with $|z|>1$, we define
\begin{align} \label{A(z)-def}
\mathbf{A}(z; \alpha) :=&\, \sum_{\ell=1}^{\infty}\tfrac{(-1)^{\ell+1}}{\ell}\tr\big\{(\Lambda\Gamma)^\ell\big\} 
\end{align}
where 
\begin{gather*}
\Lambda(z;\alpha) := \alpha (S-z^{-1})^{-1}, \qquad \Gamma(z;\alpha)  := \beta (z-S)^{-1},
\end{gather*}
and $S$ denotes the shift operator on $\ell^2(\Z)$ given by $\big(S[f]\big)_n=f_{n+1}$.

In view of these definitions of $\Lambda$ and $\Gamma$, we see that  \eqref{A(z)-def} has the form of a power series in $\alpha$ and $\bar \alpha$.  The convergence of this series will be addressed in Proposition~\ref{P:F is}.  In understanding what the conservation of $\mathbf A$ expresses about solutions to \eqref{AL}, it is useful to compute the quadratic term $\mathbf A^{[2]}$ exactly:
$$
\mathbf A^{[2]}(z; \alpha)=\tr\big\{\Lambda\Gamma \big\}=\pm\int_{-\pi}^{\pi} \frac{ |\widehat{\alpha}(\theta)|^2}{z^2e^{-\ui\theta}-1}\,\frac{d\theta}{2\pi}.
$$
Although this quantity is not coercive, one may remedy this by taking linear combinations with $M$.  The particular linear combination that will be useful to us~is 
\begin{align}\label{F defn}
\mathbf G(\kappa h;\alpha) = \pm \tfrac{2}{e^{4\kappa h} +1} M(\alpha) \mp \tanh(2\kappa h) \Re\Bigl[ \mathbf{A}(e^{\kappa h}; \alpha) + \mathbf{A}(\ui e^{\kappa h}; \alpha) \Bigr],
\end{align}
which has quadratic part
\begin{align}\label{F[2]}
\mathbf G^{[2]}(\kappa h;\alpha) = \int_{-\pi}^{\pi} \frac{\sin^2(\theta)\, |\widehat{\alpha}(\theta)|^2}{\sinh^2(2\kappa h) + \sin^2(\theta)}\,\frac{d\theta}{2\pi}.
\end{align}

\begin{proposition}\label{P:F is}
Let $\alpha$ be the solution to \eqref{AL} with initial data \eqref{E:initial data} and $h$ satisfying \eqref{small h}. Then there exists $\kappa_0>0$ depending only on $\|\psi_0\|_{L_x^2}$ and $\|\phi_0\|_{L_x^2}$ so that for all $t\in \R$ and all $\kappa\geq\kappa_0$,  the series defining $\mathbf G(\kappa h;\alpha(t))$ converges and is independent of $t$.  Moreover,
\begin{align}\label{E:F is F2}
\Bigl| \mathbf G(\kappa h;\alpha(t)) - \mathbf G^{[2]}(\kappa h;\alpha(t)) \Bigr|\lesssim \sinh^{-1}(\kappa h) \|\alpha(0)\|_{\ell^2_n}^4,
\end{align}
with the implicit constant independent of $h$.
\end{proposition}

\begin{proof}
It was observed in \cite[Lemma 5.2]{Harrop-Griffiths.Killip.Visan2021} that $\Lambda$ and $\Gamma$ are Hilbert--Schmidt operators with 
\begin{gather}\label{AL HS}
\bigl\| \Lambda(z;\alpha) \bigr\|_{\j_2}^2 = \tfrac{|z|^2}{|z|^2 - 1} \|\alpha\|_{\ell^2_n}^2
\qtq{and}
\bigl\| \Gamma(z;\alpha)   \bigr\|_{\j_2}^2  = \tfrac{1}{|z|^2 - 1} \|\alpha\|_{\ell^2_n}^2.
\end{gather}
Moreover, by \cite[Theorem 5.1]{Harrop-Griffiths.Killip.Visan2021}, $\mathbf{A}(z;\alpha)$ converges and is conserved under the \eqref{AL} flow provided that
$$
\tfrac{|z|}{|z|^2-1}\|\alpha\|_{\ell^2_n}^2<1.
$$
Combining this \eqref{8:52}, \eqref{A(z)-def}, and Proposition~\ref{P:mass}, we may estimate
\begin{align*}
\text{LHS}\eqref{E:F is F2}
&\leq \tanh(2\kappa h)\sum_{\ell\geq 2} \tfrac1\ell\bigl\| \Lambda(e^{\kappa h};\alpha(t)) \bigr\|_{\j_2}^\ell\bigl\| \Gamma(e^{\kappa h};\alpha(t)) \bigr\|_{\j_2}^\ell \\
&\quad + \tanh(2\kappa h)\sum_{\ell\geq 2} \tfrac1\ell \bigl\| \Lambda(\ui e^{\kappa h};\alpha(t)) \bigr\|_{\j_2}^\ell\bigl\| \Gamma(\ui e^{\kappa h};\alpha(t)) \bigr\|_{\j_2}^\ell\\
&\quad +  \tfrac{2}{e^{4\kappa h} + 1}\sum_{\ell\geq 2}\tfrac{1}{\ell}\nm{\ah(t)}_{\l^2_n}^{2\ell}\\
&\leq \tanh(2\kappa h)\sum_{\ell\geq 2} \Bigl(\tfrac{e^{\kappa h}}{e^{2\kappa h}-1}  \|\alpha(t)\|_{\ell^2_n}^2\Bigr)^{\ell} + \tfrac{1}{\cosh(2\kappa h)}\|\alpha(t)\|_{\ell^2_n}^4\lesssim \text{RHS}\eqref{E:F is F2},
\end{align*}
provided $\kappa\geq \kappa_0$ for some $\kappa_0$ depending only on the $L^2_x$ norms of $\phi_0,\psi_0$.
\end{proof}

\section{Strichartz Estimates}\label{Sec:Strichartz}

Our main goal in this section is to prove certain Strichartz estimates for $\alpha$, $\psi^h$, and $\phi^h$ that are uniform in the parameter $h$.  We start by recording the Strichartz estimates for the continuum and discrete Schr\"odinger propagators.

\begin{proposition}[Strichartz estimates for $e^{\ui t\D}$; \cite{GV}]\label{P:Strichartz}
Let $(q,r)$ and $(\tilde q, \tilde r)$ be two pairs satisfying
\begin{align*}
q,r, \tilde q, \tilde r\geq2 \quad\text{and} \quad \tfrac{2}{q}+\tfrac{1}{r}=\tfrac{1}{2}=\tfrac{2}{\tilde q}+\tfrac{1}{\tilde r}.
\end{align*}
If $\psi$ is a solution to the Schr\"odinger equation
\begin{align*}
\ui\dt\psi=-\D \psi+F
\end{align*}
with initial data $\psi_0\in L^2_x(\r)$, then
\begin{align*}
\nm{\psi}_{L_t^qL^r_x(\R\times\r)}\ls\nm{\psi_0}_{L^2_x(\r)}+\nm{F}_{L_t^{\tilde q'}L_x^{\tilde r'}(\R\times\r)}.
\end{align*}
\end{proposition}

As shown in \cite[Theorem~3]{Stefanov.Kevrekidis2005}, Strichartz estimates for the discrete Schr\"odinger propagator $e^{\ui t\DD}$, where $\DD$ denotes the discrete Laplacian
$$
[\DD \alpha]_n= \alpha_{n-1} - 2\alpha_n + \alpha_{n+1},
$$
can be derived via the same techniques used in the continuum case, \cite{KeelTao}.  As the discrete dispersion relation has inflection points (unlike in the continuum case), the  estimates are more closely related to those familiar from the Airy propagator:

\begin{proposition}[Strichartz estimates for $e^{\ui t\DD}$; \cite{Stefanov.Kevrekidis2005}] \label{P:discrete Strichartz}
Let $(q,r)$ and $(\tilde q, \tilde r)$ be two pairs satisfying
\begin{align*}
q,r, \tilde q, \tilde r\geq2 , \qquad \tfrac{1}{q}+\tfrac{1}{3r}\leq \tfrac{1}{6}, \qquad\text{and} \qquad \tfrac{1}{\tilde q}+\tfrac{1}{3\tilde r}\leq \tfrac{1}{6}.
\end{align*}
If $\alpha$ is a solution to the discrete Schr\"odinger equation
\begin{align*}
\ui\dt\ah_n=-(\DD\ah)_n+F_n
\end{align*}
with initial data $\ah_0\in\l^2_n(\z)$, then
\begin{align*}
\nm{\ah}_{L_t^q\l^r_n(\R\times\z)}\ls\nm{\ah_{0}}_{\l^2_n(\z)}+\bnm{F}_{L_t^{\tilde q'}\l_n^{\tilde r'}(\R\times\z)}.
\end{align*}
\end{proposition}

However, if we project to frequencies away from the inflection points, we recover the same dispersive decay we find in the continuum.  Consequently, we have

\begin{proposition}[Frequency-localized Strichartz estimates for $e^{\ui t\DD}$]\label{P:loc Strichartz}
Let $(q,r)$ and $(\tilde q, \tilde r)$ be two pairs satisfying
\begin{align*}
q,r, \tilde q, \tilde r\geq2 \quad\text{and} \quad \tfrac{2}{q}+\tfrac{1}{r}=\tfrac{1}{2}=\tfrac{2}{\tilde q}+\tfrac{1}{\tilde r}.
\end{align*}
The solution to the discrete Schr\"odinger equation
\begin{align*}
\ui\dt\ah_n=-(\DD\ah)_n+F_n
\end{align*}
with initial data $\ah_0\in\l^2_n(\z)$ satisfies
\begin{align*}
\nm{P\ah}_{L_t^q\l^r_n(\R\times\z)}\ls\nm{P\ah_{0}}_{\l^2_n(\z)}+\bnm{PF}_{L_t^{\tilde q'}\l_n^{\tilde r'}(\R\times\z)},
\end{align*}
where $P=P_{\frac34}$ is the projection operator defined in Proposition~\ref{P:suppression}.
\end{proposition}

Our main result in this section is the following:

\begin{proposition} \label{Prop:strichartz-bounds} Let $\alpha$ be the solution to \eqref{AL} with initial data \eqref{E:initial data} and $h$ satisfying \eqref{small h}.
Then for any $T>0$ we have 
\begin{align}\label{2:15}
    \nm{\ah}_{L^6_t\l^6_n([-h^{-2}T, h^{-2}T]\times\Z)} + \nm{\ah}_{L^4_t\l^\infty_n ([-h^{-2}T, h^{-2}T]\times\Z)} \lesssim_T \nm{\ah_0}_{\l^2_n},
\end{align}
where the implicit constant is independent of $h$.  Consequently,
\begin{align*}
\bnm{\ph}_{L^6_{t,x}([-T,T]\times\r)}+ \bnm{\phi^h}_{L^6_{t,x}([-T,T]\times\r)}\lesssim_T 1 \qquad\text{uniformly in h},
\end{align*}
where $\psi^h$ and $\phi^h$ are as defined in \eqref{psi phi def}.
\end{proposition}

\begin{proof}
First note that by H\"older's inequality and Proposition~\ref{P:mass},
\begin{align}\label{2:16}
\nm{\ah}_{L^6_t\l^6_n([-h^{-2}T, h^{-2}T]\times\Z)} &+ \nm{\ah}_{L^4_t\l^\infty_n ([-h^{-2}T, h^{-2}T]\times\Z)} \notag\\
&\lesssim (h^{-2}T)^{\frac16} \|\alpha\|_{L_t^\infty \ell_n^2(\R\times\Z)} +(h^{-2}T)^{\frac14} \|\alpha\|_{L_t^\infty \ell_n^2(\R\times\Z)} \notag\\
&\lesssim \bigl[(h^{-2}T)^{\frac16}+(h^{-2}T)^{\frac14} \bigr]\|\alpha(0)\|_{\ell_n^2}. 
\end{align}
In order to eliminate the dependence on $h$ in the inequality above, we will run a bootstrap argument combined with Strichartz estimates.

Let $P=P_{\frac1{100}}$ and $\widetilde P= P_{\frac34}$ denote the sharp Fourier cutoffs introduced in \eqref{P defn}.  Writing $F(\alpha)$ for the nonlinearity in \eqref{AL} and exploiting that  $\widetilde P [F(P\alpha)]= F(P\alpha) $, we obtain the following Duhamel representation of the solution:
\begin{align*}
\ah(t)&=e^{\ui t\DD}\ah(0)-\ui\int_0^{t}e^{\ui (t-\tau)\DD}F\bigl(\alpha(\tau)\bigr)\,d\tau\\
&=e^{\ui t\DD}\widetilde P\ah(0) +e^{\ui t\DD}(1-\widetilde P)\ah(0) -\ui\int_0^{t} \widetilde P e^{\ui (t-\tau)\DD}F\bigl(P\alpha(\tau)\bigr)\,d\tau \\
&\qquad -\ui\int_0^{t}e^{\ui (t-\tau)\DD}\Bigl[F\bigl(\alpha(\tau)\bigr)-F\bigl(P\alpha(\tau)\bigr)\Bigr]\,d\tau.
\end{align*}

With a view to closing a bootstrap argument, we will estimate these terms on the time interval $[-h^{-2}T_0, h^{-2}T_0]$ with $T_0\leq T$, which may later be chosen small.  (Recall that $T$ was arbitrarily large.)  For two of the terms, the presence of $\widetilde P$ allows us to employ Proposition~\ref{P:loc Strichartz} and so treat both spacetime norms of interest simultaneously.  First,
\begin{align}\label{3:05}
    \bigl\| e^{\ui t\DD}\widetilde P\ah(0) \bigr\|_{L^6_t\l^6_n \cap L^4_t \ell^\infty_n}\ls\|\widetilde P\ah(0)\|_{\l_n^2} \lesssim \|\alpha(0)\|_{\ell_n^2}.
\end{align}
Secondly, using Proposition~\ref{P:mass}, we have
\begin{align} \label{3:07}
\nm{\int_0^{t}\widetilde P e^{\ui (t-\tau)\DD}F\bigl(P\alpha(\tau)\bigr)\,d\tau}_{L^6_t\l^6_n\cap L^4_t \ell^\infty_n}
&\ls \bnm{F(P\ah)}_{L^{\frac{6}{5}}_t\l^{\frac{6}{5}}_n}\no\\
&\ls \left(h^{-2}T_0\right)^{\frac{1}{2}}\nm{\ah}_{L^{\infty}_t\l_n^2}\nm{\ah}_{L^{6}_t\l_n^6}^2 \no\\
&\ls \left(h^{-2}T_0\right)^{\frac{1}{2}}\nm{\ah(0)}_{\l_n^2}\nm{\ah}_{L^{6}_t\l_n^6}^2.
\end{align}

Turning to the remaining terms in the Duhamel expansion, we first narrow our focus to just the $L^6_t\ell^6_n$ norm.  By Propositions~\ref{P:discrete Strichartz} and~\ref{P:suppression},
\begin{align}\label{3:06}
\nm{e^{\ui t\DD}(1-\widetilde P)\ah(0)}_{L^6_t\l^6_n}&\lesssim (h^{-2}T_0)^{\frac{1}{18}} \nm{e^{\ui t\DD}(1-\widetilde P)\ah(0)}_{L^9_t\l^6_n}\notag\\
& \lesssim (h^{-2}T_0)^{\frac{1}{18}} \|(1-\widetilde P)\ah(0)\|_{\l_n^2} \notag\\
&\lesssim (h^{-2}T_0)^{\frac{1}{18}}\bigl(Nh+h^{\frac12}\bigr) \|\ah(0)\|_{\l_n^2}.
\end{align}
Proceeding in a parallel fashion, we find
\begin{align}\label{3:08}
&\nm{\int_0^{t}e^{\ui (t-\tau)\DD}\Bigl[F\bigl(\alpha(\tau)\bigr)-F\bigl(P\alpha(\tau)\bigr)\Bigr]\,d\tau}_{L^6_t\l^6_n}\no\\
&\qquad\ls(h^{-2}T_0)^{\frac{1}{18}}\nm{\int_0^{t}e^{\ui (t-\tau)\DD}\Bigl[F\bigl(\alpha(\tau)\bigr)-F\bigl(P\alpha(\tau)\bigr)\Bigr]\,d\tau}_{L^9_t\l^6_n}\no\\
&\qquad \ls (h^{-2}T_0)^{\frac{1}{18}}\bnm{F(\ah)-F(P\ah)}_{L^{\frac{9}{8}}_t\l^{\frac{6}{5}}_n}\no\\
&\qquad \ls(h^{-2}T_0)^{\frac{11}{18}}\bnm{(I-P)\ah}_{L^{\infty}_t\l^{2}_n}\nm{\ah}_{L^{6}_t\l_n^6}^2\no\\
&\qquad \ls(h^{-2}T_0)^{\frac{11}{18}} \bigl(Nh+h^{\frac12}\bigr) \|\ah(0)\|_{\l_n^2} \nm{\ah}_{L^{6}_t\l_n^6}^2 \no\\
&\qquad \ls(h^{-2}T_0)^{\frac{11}{18}} \bigl(Nh+h^{\frac12}\bigr) h^\frac12 \bigl[\|\psi_0\|_{L^2} + \|\phi_0\|_{L^2} \bigr] \nm{\ah}_{L^{6}_t\l_n^6}^2 .
\end{align}
The last step here was an application of \eqref{E:mass bound}.

Combining \eqref{3:05} through \eqref{3:08}, recalling that $N\leq h^{-\frac79}$, and using Proposition~\ref{P:mass}, we deduce that
\begin{align*}
&\nm{\ah}_{L^6_t\l^6_n([-h^{-2}T_0, h^{-2}T_0]\times\Z)}\\
&\qquad \qquad\ls \big(1+T_0^{\frac{1}{18}}\big)\nm{\ah(0)}_{\l^2_n}+ h^{-\frac{1}{2}}\big(T_0^{\frac{1}{2}}+T_0^{\frac{11}{18}}\big)\nm{\ah}_{L^{6}_t\l_n^6([-h^{-2}T_0, h^{-2}T_0]\times\Z)}^2,
\end{align*}
where the implicit constant depends only on $\psi_0$ and $\phi_0$.   Taking $T_0$ sufficiently small, a bootstrap argument yields
\begin{align*}
&\nm{\ah}_{L^6_t\l^6_n([-h^{-2}T_0, h^{-2}T_0]\times\Z)}\ls \nm{\ah(0)}_{\l^2_n}.
\end{align*}
Note that \eqref{2:16} guarantees that the quantity being bootstrapped is initially finite.

Iterating this argument $(1+ \frac T{T_0})$ many times, we conclude that
\begin{align}\label{T6}
\nm{\ah}_{L^6_t\l^6_n([-h^{-2}T, h^{-2}T]\times\Z)}\ls_T \nm{\ah(0)}_{\l^2_n}.
\end{align}
In particular, in view of Lemma~\ref{Lem:norm-relation} and Proposition~\ref{P:mass}, this yields
\begin{align*}
    \nm{\ph}_{L^6_{t,x}([-T,T]\times\R)} +\nm{\phi^h}_{L^6_{t,x}([-T,T]\times\R)}\ls h^{-\frac{1}{2}}\nm{\alpha}_{L^6_t\l^6_n([-h^{-2}T, h^{-2}T]\times\Z)}\ls_T 1.
\end{align*}

We turn now to the $L^4_tL^{\infty}_x$ norm, which we will treat using \eqref{T6} and so not need to argue via bootstrap.  Henceforth, all norms will be taken over the full time interval $[-h^{-2}T, h^{-2}T]$.

Mimicking \eqref{3:06}, we find
\begin{align}\label{4:06}
\nm{e^{\ui t\DD}(1-\widetilde P)\ah(0)}_{L^4_t\l^{\infty}_n}&\lesssim (h^{-2}T)^{\frac{1}{12}} \nm{e^{\ui t\DD}(1-\widetilde P)\ah(0)}_{L^6_t\l^\infty_n}\notag\\
& \lesssim (h^{-2}T)^{\frac{1}{12}} \|(1-\widetilde P)\ah(0)\|_{\l_n^2} \notag\\
&\lesssim (h^{-2}T)^{\frac{1}{12}}\bigl(Nh+h^{\frac12}\bigr) \|\ah(0)\|_{\l_n^2}.
\end{align}
Likewise, paralleling \eqref{3:08}, we find
\begin{align}\label{4:08}
&\nm{\int_0^{t}e^{\ui (t-\tau)\DD}\Bigl[F\bigl(\alpha(\tau)\bigr)-F\bigl(P\alpha(\tau)\bigr)\Bigr]\,d\tau}_{L^4_t\l^{\infty}_n}\no\\
&\qquad\ls(h^{-2}T_0)^{\frac{1}{12}}\nm{\int_0^{t}e^{\ui (t-\tau)\DD}\Bigl[F\bigl(\alpha(\tau)\bigr)-F\bigl(P\alpha(\tau)\bigr)\Bigr]\,d\tau}_{L^6_t\l^\infty_n}\no\\
&\qquad \ls (h^{-2}T_0)^{\frac{1}{12}}\bnm{F(\ah)-F(P\ah)}_{L^{\frac98}_t\l^{\frac65}_n}\no\\
&\qquad \ls(h^{-2}T_0)^{\frac{23}{36}} \bigl(Nh+h^{\frac12}\bigr) h^\frac12 \bigl[\|\psi_0\|_{L^2} + \|\phi_0\|_{L^2} \bigr] \nm{\ah}_{L^{6}_t\l_n^6}^2 .
\end{align}
Combining \eqref{3:05}, \eqref{3:07},  \eqref{4:06}, \eqref{4:08}, \eqref{T6}, and the fact that $N\lesssim h^{-\frac{13}{18}}$ yields
\begin{align*}
\nm{\alpha}_{L^4_t\l^\infty_n([-h^{-2}T, h^{-2}T]\times\Z)}\ls_T  \|\ah(0)\|_{\l_n^2}
\end{align*}
and so completes the proof of the proposition.
\end{proof}


\section{Precompactness in $C_tL^2_x$}\label{sec:precompactness}

This section is dedicated to the proof of the following precompactness result:

\begin{theorem} \label{T:precompactness}
Let $\alpha$ be the solution to \eqref{AL} with initial data \eqref{E:initial data} and $h$ satisfying \eqref{small h} and let $\psi^h$ and $\phi^h$ be as in \eqref{psi phi def}. For $T>0$ fixed, the two families of functions $\Psi=\big\{\psi^h(t,x):0<h\leq h_0 \big\}$ and $\Phi=\big\{\phi^h(t,x):0<h\leq h_0 \big\}$ are precompact in $C([-T,T];L^2_x(\r))$.
\end{theorem}

By a generalization of the Arzel\`a--Ascoli theorem due to M. Riesz \cite{Riesz1933}, precompactness of the two families will follow once we establish uniform boundedness, equicontinuity, and tightness properties. Specifically, we will demonstrate:\\[2mm]
\noindent\emph{Uniform boundedness}: there exists $C>0$ such that
\begin{align}\label{unif bdd}
\sup_{0<h\leq h_0}\Bigl[\|\psi^h(t)\|_{L^\infty_t L^2_x([-T,T]\times\R)}+\|\phi^h(t)\|_{L^\infty_t L^2_x([-T,T]\times\R)}\Bigr]\leq C.
\end{align}
\noindent\emph{Equicontinuity}: for any $\ep>0$, there exists $\delta>0$ so that whenever $\abs{s}+\abs{y}<\delta$,
\begin{align} \label{equi}
\|\psi^h(t+s,x+y)-\psi^h(t,x)\|_{L^2_x}+\|\phi^h(t+s,x+y)-\phi^h(t,x)\|_{L^2_x}< \ep
\end{align}
uniformly for $t\in[-T,T]$ (with $t+s\in [-T,T]$) and for $0<h\leq h_0$.\\[2mm]
\noindent\emph{Tightness}: for any $\ep>0$, there exists $R>0$ such that
\begin{align} \label{tightness}
\sup_{0<h\leq h_0}  \sup_{|t|\leq T}\int_{\abs{x}\geq R}\babs{\ph(t,x)}^2+\babs{\phi^h(t,x)}^2\, dx<\ep.
\end{align}

As $h\mapsto(\ph,\phi^h)$ defines a continuous mapping from $(0,h_0]$ to $C([-T,T];L^2_x(\r))$, these three conditions automatically hold on any interval of the form $[h_1,h_0]$.  Correspondingly, it suffices to prove \eqref{equi} and \eqref{tightness} only for $0<h\leq h_1$ where $h_1$ may depend on $\eps$. 

Combining Lemma~\ref{Lem:norm-relation} and Proposition~\ref{P:mass}, we find
\begin{align*}
\|\psi^h(t)\|_{L^\infty_t L^2_x([-T,T]\times\R)}&+\|\phi^h(t)\|_{L^\infty_t L^2_x([-T,T]\times\R)} \\
&\lesssim h^{-\frac12} \|\alpha\|_{L_t^\infty \ell_n^2([-h^{-2}T, h^{-2}T]\times\Z)} \lesssim 1,
\end{align*}
uniformly for $0<h\leq h_0$, which settles \eqref{unif bdd}.

We turn now to the equicontinuity property, starting with equicontinuity in the spatial variable.  By the uniform boundedness property \eqref{unif bdd} and Plancherel, this is equivalent to tightness on the Fourier side, that is, for any $\ep>0$ there exist $\kappa>0$ such that
\begin{align}\label{equi 2}
\sup_{|t|\leq T}\|P_{|\xi|\geq \kappa} \psi^h(t)\|_{L_x^2} +\|P_{|\xi|\geq \kappa} \phi^h(t)\|_{L_x^2}<\ep
\end{align}
uniformly for $0<h\leq h_0$.

A straightforward computation using \eqref{psi phi def} shows that for $\kappa h<\frac\pi2$ we have
\begin{align*}
\|P_{|\xi|\geq \kappa} \psi^h(t)\|_{L_x^2}^2 +\|P_{|\xi|\geq \kappa} \phi^h(t)\|_{L_x^2}^2
&= \int_{\kappa h\leq |h\xi| \leq \pi - \kappa h} \bigl| \widehat{\alpha}(h^{-2}t, h\xi) \bigr|^2\, \tfrac{d\xi}{2\pi}\\
&\leq \tfrac Ch \int_{\kappa h\leq |\theta| \leq \pi - \kappa h}\tfrac{\sin^2(\theta)\, |\widehat{\alpha}(h^{-2}t,\theta)|^2}{\sinh^2(2\kappa h) + \sin^2(\theta)}\,\tfrac{d\theta}{2\pi}\\
&\leq \tfrac Ch \mathbf G^{[2]}\bigl(\kappa h;\alpha(h^{-2}t)\bigr)
\end{align*}
for a universal constant $C>0$ and $ \mathbf G^{[2]}$ as defined in \eqref{F[2]}.  Using Proposition~\ref{P:F is} followed by \eqref{data supp} and Proposition~\ref{P:mass},  we may thus estimate
\begin{align*}
\sup_{|t|\leq T} &\Bigl[\|P_{|\xi|\geq \kappa} \psi^h(t)\|_{L_x^2}^2 +\|P_{|\xi|\geq \kappa} \phi^h(t)\|_{L_x^2}^2\Bigr]\\
&\quad\lesssim h^{-1}  \bigl[\mathbf G^{[2]}\bigl(\kappa h, \alpha(0)\bigr) +  \sinh^{-1}(\kappa h)\|\alpha(0)\|_{\ell^2_n}^4\bigr]\\
&\quad \lesssim \int_{-N}^N \tfrac{\sin^2(h\xi)}{\sinh^2(2\kappa h) + \sin^2(h\xi)}\,\bigl[ |\widehat{\psi_0}(\xi)|^2+  |\widehat{\phi_0}(\xi)|^2\bigr]\,d\xi+ \kappa^{-1}\\
&\quad\lesssim  \|P_{|\xi|\geq \sqrt\kappa}\,\psi_0\|_{L_x^2}^2 + \|P_{|\xi|\geq \sqrt\kappa}\, \phi_0\|_{L_x^2}^2
+ \kappa^{-1} \bigl[ \|\psi_0\|_{L_x^2}^2+ \|\phi_0\|_{L_x^2}^2\bigr] + \kappa^{-1}.
\end{align*}
Choosing $\kappa=\kappa(\eps)$ sufficiently large we can guarantee that \eqref{equi 2} holds whenever $0<h<h_1$ for $h_1=h_1(\eps)$.  Note that the restriction $0<h<h_1$ ensures that $\kappa h<\frac\pi2$, which allowed for the computations above.  Recall that spatial equicontinuity in the regime $h\in[h_1,h_0]$ is a consequence of the compactness of the interval $[h_1,h_0]$ and the continuity of the mapping $h\mapsto(\ph,\phi^h)$.

We now turn to the second half of \eqref{equi}, namely, equicontinuity in the time variable.  By \eqref{psi phi def} and Plancherel,
\begin{align*}
\bigl\| \psi^h(t+s)-\psi^h(t)\bigr\|_{L_x^2}^2&=\tfrac1h\bigl\| P_{|\theta|<\frac\pi2}\bigl[\alpha \bigl(h^{-2}(t+s)\bigr)-\alpha \bigl(h^{-2}(t)\bigr)\bigr]\bigr\|_{\ell_n^2}^2,\\
\bigl\| \phi^h(t+s)-\phi^h(t)\bigr\|_{L_x^2}^2 &= \tfrac1h \bigl\| P_{|\theta-\pi|<\frac\pi2}\bigl[e^{4\ui h^{-2}s}\alpha \bigl(h^{-2}(t+s)\bigr)- \alpha \bigl(h^{-2}(t)\bigr)\bigr]\bigr\|_{\ell_n^2}^2.
\end{align*}
To estimate the right-hand sides above, we will rely on Duhamel's formula
\begin{align*}
    \ah\bigl(h^{-2}(t+s)\bigr)=e^{\ui h^{-2}s\DD}\ah\bigl(h^{-2}t\bigr)-\ui\int_{h^{-2}t}^{h^{-2}(t+s)} e^{\ui [h^{-2}(t+s)-\tau]\DD}F\bigl(\alpha(\tau)\bigr)\,d\tau.
\end{align*}

Using Plancherel and evaluating the contributions of the regions $|\theta|< \kappa h$ and $\kappa h<|\theta|<\frac\pi2$ separately, we find
\begin{align*}
\bigl\| P_{|\theta|<\frac\pi2}\bigl[\alpha \bigl(h^{-2}(t+s)\bigr)&-\alpha \bigl(h^{-2}(t)\bigr)\bigr]\bigr\|_{\ell_n^2}^2\\
&=\int_{|\theta|<\frac\pi2} \bigl| e^{-4ih^{-2}s\sin^2(\frac \theta2)}-1\bigr|^2 \bigl| \widehat{\alpha}\bigl(h^{-2}t, \theta\bigr)\bigr|^2\,\tfrac{d\theta}{2\pi} \\
&\lesssim \kappa^4|s|\bigl\|\ah\bigl(h^{-2}t\bigr)\bigr\|_{\ell_n^2}^2+ h \bigl \|P_{|\xi|\geq \kappa} \psi^h(t)\bigr\|_{L_x^2}^2.
\end{align*}
In view of Proposition~\ref{P:mass} and \eqref{equi 2}, we may choose $\kappa=\kappa(\ep)$ sufficiently large, followed by $\delta=\delta(\ep)$ sufficiently small to guarantee that
\begin{align}\label{5:22}
\bigl\| P_{|\theta|<\frac\pi2}\bigl[\alpha& \bigl(h^{-2}(t+s)\bigr)-\alpha \bigl(h^{-2}(t)\bigr)\bigr]\bigr\|_{\ell_n^2}^2\leq \tfrac{\ep^2 h}{10} \quad\text{for all $|s|\leq \delta$}.
\end{align}
Arguing similarly, we find
\begin{align}\label{5:23}
\bigl\| P_{|\theta-\pi|<\frac\pi2}\bigl[e^{4\ui h^{-2}s}&\alpha \bigl(h^{-2}(t+s)\bigr)- \alpha \bigl(h^{-2}(t)\bigr)\bigr]\bigr\|_{\ell_n^2}^2\notag\\
&=\int_{|\theta|<\frac\pi2} \bigl| e^{4ih^{-2}s\sin^2(\frac {\theta+\pi}2)}-1\bigr|^2 \bigl| \widehat{\alpha}\bigl(h^{-2}t, \theta+\pi\bigr)\bigr|^2\,\tfrac{d\theta}{2\pi} \notag\\
&\lesssim \kappa^4|s|\bigl\|\ah\bigl(h^{-2}t\bigr)\bigr\|_{\ell_n^2}^2+ h \bigl \|P_{|\xi|\geq \kappa} \phi^h(t)\bigr\|_{L_x^2}^2\leq \tfrac{\ep^2 h}{10}
\end{align}
for $\kappa=\kappa(\ep)$ sufficiently large and $|s|\leq \delta=\delta(\eps)$.

It remains to estimate the contribution of the nonlinearity.  To this end, let $P=P_{\frac1{100}}$ and $\widetilde P= P_{\frac34}$ denote the sharp Fourier cutoffs introduced in \eqref{P defn}.  We will again use that $\widetilde P [F(P\alpha)]= F(P\alpha) $ to estimate the contribution of $F(P\alpha)$ using the frequency-localized Strichartz estimates of Proposition~\ref{P:loc Strichartz}, followed by Propositions~\ref{P:mass} and \ref{Prop:strichartz-bounds}:
\begin{align}\label{5:24}
\Bigl\| \int_{h^{-2}t}^{h^{-2}(t+s)} & e^{\ui [h^{-2}(t+s)-\tau]\DD}F\bigl(P\alpha(\tau)\bigr)\,d\tau\Bigr\|_{\ell_n^2}^2\notag\\
&\lesssim \bigl\| F(P\alpha)\bigr\|_{L_t^{\frac{6}{5}}\l_n^{\frac{6}{5}}([-h^{-2}t, h^{-2}(t+s)]\times\Z)}^2\notag\\
&\lesssim h^{-2}|s|\nm{\ah}_{L^6_t\l_n^6([-h^{-2}T, h^{-2}T]\times\Z)}^4\nm{\ah}_{L^{\infty}_t\l_n^2([-h^{-2}T, h^{-2}T]\times\Z)}^2
\lesssim_T  h|s| \leq \tfrac{\ep^2 h}{10},
\end{align}
provided $\delta=\delta(\ep,T)$ is chosen sufficiently small.

Finally we estimate the contribution of $F(\alpha)-F(P\alpha) $ using the discrete Strichartz estimates from Proposition~\ref{P:discrete Strichartz}, followed by Propositions~\ref{Prop:strichartz-bounds} and \ref{P:suppression}:
\begin{align}\label{5:25}
\Bigl\| \int_{h^{-2}t}^{h^{-2}(t+s)} &e^{\ui [h^{-2}(t+s)-\tau]\DD}\bigl(F(\ah)-F(P\ah)\bigr) \, d\tau\Bigr\|_{\ell_n^2}^2\notag\\
&\ls \bnm{F(\ah)-F(P\ah)}_{L_t^{\frac{9}{8}}\l_n^{\frac{6}{5}}([-h^{-2}t, h^{-2}(t+s)]\times\Z)}^2\notag\\
&\ls \bigl(h^{-2}|s|\bigr)^{\frac{10}{9}}\nm{\ah}_{L^6_t\l_n^6([-h^{-2}T, h^{-2}T]\times\Z)}^4\bnm{(I-P)\ah}_{L^{\infty}_t\l_n^2([-h^{-2}T, h^{-2}T]\times\Z)}^2\notag\\
&\ls_T \bigl(h^{-2}|s|\bigr)^{\frac{10}{9}} \bigl(Nh+h^{\frac12}\bigr)^2\nm{\ah_0}_{\l_n^2}^6 \leq \tfrac{\ep^2 h}{10},
\end{align}
since $N \leq h^{-\frac89}$ and $\delta=\delta(\ep,T)$ is chosen sufficiently small.

Collecting \eqref{5:22} through \eqref{5:25}, we conclude that for all $|s|\leq \delta=\delta(\ep,T)$,
\begin{align}\label{equi 3}
\|\psi^h(t+s)-\psi^h(t)\|_{L^2_x}+\|\phi^h(t+s)-\phi^h(t)\|_{L^2_x}< \ep,
\end{align}
uniformly for $t\in[-T,T]$ (with $t+s\in [-T,T]$) and for $0<h\leq h_0$.  This expresses equicontinuity in the time variable and combined with \eqref{equi 2} settles \eqref{equi}.

Lastly, we will demonstrate the tightness property \eqref{tightness}.  To this end, let $\chi(x)$ be a smooth cutoff function satisfying
\begin{align*}
    \chi(x)=\begin{cases} 1&: \abs{x}\leq 1\\
    0 &: \abs{x}\geq 2\end{cases}
\end{align*}
from which we build a cutoff function to large $n$ on the lattice via $\varphi_R(n):=1-\chi(nh/R)$ for some $R\geq 1$ to be chosen later.

We will also be localizing in frequency: We define $P:\ell^2_n\to\ell^2_n$ via
\begin{align}\label{P def}
\widehat {P \alpha} (\theta) = \bigl[\chi\bigl(\tfrac\theta{\kappa h}\bigr) + \chi\bigl(\tfrac{\theta-\pi}{\kappa h}\bigr)\bigr] \widehat\alpha(\theta)
	\qtq{where $-\tfrac\pi2 \leq \theta < \tfrac{3\pi}2$ and $\kappa h <\tfrac\pi4$.}
\end{align}
By Schur's test, we obtain the commutator bound
\begin{align}\label{12:26}
\bigl\| [P,\varphi_R] \bigr\|_{\ell^2_n\to\ell^2_n} \lesssim \tfrac{1}{\kappa R}.
\end{align}

We will prove tightness of the orbit of the solution to \eqref{AL}.  Specifically, we will show that for any $\ep>0$ there exists $R\geq 1$ such that 
\begin{align} \label{tightness 2}
\sup_{|t|\leq h^{-2}T}\bigl\|\varphi_R\ah(t)\bigr\|_{\ell^2_n}^2<\ep h 
\end{align}
uniformly for $0<h\leq h_0$. The tightness of $\alpha_n(t)$ transfers to $\psi^h$ and $\phi^h$ as a consequence of Lemmas~\ref{L:C} and \ref{L:tighty}.

By the equicontinuity in the space variable \eqref{equi 2}, we have  
\begin{align}\label{12:27}
\bigl\|(1-P)\ah(t)\bigr\|_{\ell^2_n}^2< \ep^4 h \quad\text{uniformly for $|t|\leq h^{-2}T$ and $0<h\leq h_0$,}
\end{align}
provided $\kappa\geq 1$ is chosen sufficiently large depending only on $\eps$.  To maintain the condition $\kappa h <\tfrac\pi4$ from \eqref{P def}, we will restrict attention to  $0<h<h_1=h_1(\ep)$~and prove \eqref{tightness 2} in this regime.  Recall that tightness in the regime $h\in[h_1,h_0]$ is an immediate consequence of the compactness of the interval $[h_1,h_0]$ and the continuity of the mapping $h\mapsto \alpha$ with $\alpha$ being the solution to \eqref{AL} with initial data as in~\eqref{E:initial data}.

In view of \eqref{12:27} and the preceding discussion, \eqref{tightness 2} will thus follow from the statement that for any $\ep>0$ there exists $R\geq 1$ such that 
\begin{align} \label{tightness 3}
\sup_{|t|\leq h^{-2}T}\bigl\|\varphi_RP\ah(t)\bigr\|_{\ell^2_n}^2<\ep h \quad\text{uniformly for $0<h<h_1$.}
\end{align}

From \eqref{AL} we have
\begin{align}\label{12:40}
\partial_t \bigl\|\varphi_R P\ah(t)\bigr\|_{\ell^2_n}^2 = -2\Im \sum_n \varphi_R^2(n) \overline{P\alpha_n(t)} \cdot P\bigl\{ (\DD\ah)_n(t)- F_n[\ah(t)] \bigr\}.
\end{align}
To control the contribution of the quadratic term above we rewrite
\begin{align*}
 2&\Im\sum_n \varphi_R^2(n) P\beta_n(t) \cdot (\DD P\ah)_n(t)\\
&= 2\Im\sum_n \varphi_R^2(n) P\beta_n(t)\cdot  P\alpha_{n+1}(t) - 2\Im\sum_n \varphi_R^2(n) P\alpha_n(t)\cdot  P\beta_{n-1}(t)\\
&= 2\Im\sum_n \bigl[\varphi_R^2(n)-\varphi_R^2(n+1)\bigr]  P\beta_n(t)\cdot  P\alpha_{n+1}(t)\\
&= \Im\sum_n \bigl[\varphi_R^2(n)-\varphi_R^2(n+1)\bigr] P\beta_n(t) \cdot P\bigl[\alpha_{n+1}(t)-\alpha_{n-1}(t)\bigr]\\
&\quad +\Im\sum_n\bigl[2\varphi_R^2(n)-\varphi_R^2(n+1)-\varphi_R^2(n-1)\bigr] P\beta_n(t) \cdot P\alpha_{n-1}(t).
\end{align*}
By Plancherel, \eqref{P def}, and Proposition~\ref{P:mass}, 
\begin{align*}
\bigl\|P\bigl[\alpha_{n+1}(t)-\alpha_{n-1}(t)\bigr] \bigr\|_{\l_n^2}
&\lesssim \bigl\|\bigl[\chi\bigl(\tfrac\theta{\kappa h}\bigr)+ \chi\bigl(\tfrac{\theta-\pi}{\kappa h}\bigr) \bigr](e^{\ui\theta}-e^{-\ui\theta})\widehat\alpha(t,\theta) \bigr\|_{L_\theta^2}\\
&\ls \kappa h\nm{\ah(t)}_{\l_n^2} \ls \kappa h\nm{\ah(0)}_{\l_n^2}.
\end{align*}
Thus, using Proposition~\ref{P:mass} we may estimate
\begin{align}\label{12:41}
\Bigl|2\Im\sum_n & \varphi_R^2(n) P\beta_n(t) \cdot (\DD P\ah)_n(t) \Bigr|\notag\\
&\lesssim \kappa h\bigl\|\varphi_R^2(n)-\varphi_R^2(n+1)\bigr\|_{\ell_n^\infty}\|\ah(0)\|_{\l_n^2}^2\notag\\
&\quad+ \bigl\|2\varphi_R^2(n)-\varphi_R^2(n+1)-\varphi_R^2(n-1)\bigr\|_{\ell_n^\infty}\|\ah(0)\|_{\l_n^2}^2\notag\\
&\lesssim \bigl(\tfrac{\kappa h^2}{R} + \tfrac{h^2}{R^2} \bigr)\|\ah(0)\|_{\l_n^2}^2.
\end{align}

We turn now to the contribution of the nonlinearity and decompose $\alpha = P\alpha + (1-P)\alpha$.  The contribution of the nonlinearity containing $(1-P)\alpha$ can be estimated using \eqref{12:27}, as follows:
\begin{align}\label{12:42}
\Bigl| \Im \sum_n \varphi_R^2(n) P\beta_n(t) \cdot P\Bigl\{&(1-P)\ah_n(t)\cdot \beta_n(t)\bigl[\ah_{n+1}(t) + \ah_{n-1}(t)\bigr] \Bigr\}  \Bigr|\notag\\
&\lesssim \|(1-P)\ah(t)\|_{\ell^2_n} \|\ah(t)\|_{\ell^6_n}^3\lesssim \eps^2 h^{\frac12}  \|\ah(t)\|_{\ell^6_n}^3.
\end{align}
Using \eqref{12:26} and Proposition~\ref{P:mass}, we bound the contribution of the remaining term via
\begin{align}\label{12:43}
\Bigl| \Im \sum_n \varphi_R^2(n)& P\beta_n(t)\cdot P\Bigl\{P\ah_n(t)\cdot \beta_n(t)\bigl[\ah_{n+1}(t) + \ah_{n-1}(t)\bigr] \Bigr\}  \Bigr|\notag\\
&\lesssim \bigl\|\varphi_R P\alpha(t)\bigr\|_{\ell^2_n}^2\|\alpha(t)\|_{\ell^\infty_n}^2 +  \bigl\|\varphi_R P\alpha(t)\bigr\|_{\ell^2_n} \bigl\| [P,\varphi_R] \bigr\|_{\ell^2_n\to\ell^2_n}\|\ah(t)\|_{\ell^6_n}^3\notag\\
&\lesssim \bigl\|\varphi_R P\alpha(t)\bigr\|_{\ell^2_n}^2\|\alpha(t)\|_{\ell^\infty_n}^2 + \tfrac{h^{\frac12} }{\kappa R}\|\ah(t)\|_{\ell^6_n}^3.
\end{align}

Combining \eqref{12:40} through \eqref{12:43}, we find that for $\kappa h <\tfrac\pi4$ with $\kappa\geq 1$ large depending only on $\eps$,
\begin{align*}
\partial_t \bigl\|\varphi_R P\ah(t)\bigr\|_{\ell^2_n}^2&\lesssim \tfrac{\kappa h^2}{R}\|\ah(0)\|_{\l_n^2}^2 + \bigl(\eps^2  +\tfrac{1}{\kappa R}\bigr) h^{\frac12}\|\ah(t)\|_{\ell^6_n}^3
+ \bigl\|\varphi_R P\alpha(t)\bigr\|_{\ell^2_n}^2\|\alpha(t)\|_{\ell^\infty_n}^2.
\end{align*}
By Gronwall and Propositions~\ref{Prop:strichartz-bounds} and \ref{P:mass}, this yields
\begin{align*}
\sup_{|t|\leq h^{-2}T}\bigl\| &\varphi_R  P\ah(t)\bigr\|_{\ell^2_n}^2\\
&\lesssim \Bigl[ \bigl\|\varphi_R P\ah(0)\bigr\|_{\ell^2_n}^2 + \tfrac{\kappa h^2}{R}\|\ah(0)\|_{\l_n^2}^2 h^{-2}T + \bigl(\eps^2  +\tfrac{1}{\kappa R}\bigr) h^{\frac12}\|\ah\|_{L_t^6\ell^6_n}^3(h^{-2}T)^{\frac12}\Bigr]\\
&\qquad \qquad \times\exp\Bigl\{ C\|\alpha\|_{L^4_t\ell^\infty_n}^2(h^{-2}T)^{\frac12} \Bigr\}\\
&\lesssim_T \bigl\|\varphi_R P\ah(0)\bigr\|_{\ell^2_n}^2 +\tfrac{\kappa T}{R} h + \bigl(\eps^2  +\tfrac{1}{\kappa R}\bigr)T^{\frac12} h,
\end{align*}
where all spacetime norms are over $[-h^{-2}T, h^{-2}T]\times\Z$.   The right-hand side here can be made smaller that $\eps h$ by first choosing $\kappa$ large (to ensure \eqref{12:27}) and then choosing $R$ sufficiently large.  Note that by Lemma~\ref{Lem:norm-relation}, monotone convergence, and the $L^2$-boundedness of the Hardy--Littlewood maximal operator $\mathcal M$, we have
$$
\limsup_{R\to\infty}  h^{-1} \bigl\| \varphi_R P\ah(0)\bigr\|_{\ell^2_n}^2
	\lesssim \limsup_{R\to\infty} \bigl\| [1-\chi(\tfrac xR)] [ \mathcal M\psi_0 + \mathcal M\phi_0] \bigr\|_{L^2} = 0.
$$
This completes the proof of \eqref{tightness 3} and so that of \eqref{tightness 2}.

\section{Convergence of the flows}\label{Sec:Convergence}

As a consequence of Theorem~\ref{T:precompactness}, every sequence $h_n\rt 0$ admits a subsequence $h_{n_j}\rt 0$ such that $(\pe^{h_{n_j}},\phi^{h_{n_j}})$ converges in $C([-T,T];L^2_x(\r))$ to some $(\pe,\phi)$.   As a first step toward proving Theorem~\ref{T:main}, we will show that all such subsequential limits are solutions to \eqref{NLS system} with initial data $(\psi_0,\phi_0)$ and satisfy certain spacetime bounds.  For notational simplicity, we will omit the subscripts on $h$ in what follows.
 
\begin{proposition}\label{P:int eq}
Let $\alpha$ be the solution to \eqref{AL} with initial data \eqref{E:initial data} and $h$ satisfying \eqref{small h} and let $\psi^h$ and $\phi^h$ be as in \eqref{psi phi def}. Let $\psi,\phi \in C([-T,T];L^2_x(\R))$ so~that 
\begin{align}\label{convg}
\psi^h \to \psi \qtq{and} \phi^h \to\phi \qtq{in} C([-T,T];L^2_x(\R))
\end{align}
along some sequence of $h\to 0$.  Then $\psi,\phi \in L^6_{t,x}([-T,T]\times\R)$ and for any $|t|\leq T$ we have
\begin{align}
    \pe(t)&=e^{\ui t\D}\pe_0\mp2\ui\int_0^te^{\ui (t-s)\D}\abs{\pe(s)}^2\pe(s)\,ds,\label{duhamel 1}\\
    \phi(t)&=e^{-\ui t\D}\phi_0\pm2\ui\int_0^te^{-\ui (t-s)\D}\abs{\phi(s)}^2\phi(s)\,ds. \label{duhamel 2}
\end{align}
\end{proposition}

\begin{proof}
As $\psi^h$ and $\phi^h$ converge in $C([-T,T];L^2_x(\R))$, they converge distributionally on $[-T,T];\times\R$.  Consequently, by Proposition~\ref{Prop:strichartz-bounds},  $\psi$ and $\phi$ satisfy
\begin{align}\label{6 bdd}
\|\psi\|_{L^6_{t,x}([-T,T]\times\R)} + \|\phi\|_{L^6_{t,x}([-T,T]\times\R)}\lesssim_T 1.
\end{align}

To prove \eqref{duhamel 1} and \eqref{duhamel 2}, our starting point is the Duhamel formula satisfied by the solution $\ah$ of \eqref{AL}: for any $|t|\leq T$,
\begin{align}\label{ah duh}
    \ah_n\bigl(h^{-2}t\bigr)=e^{\ui h^{-2}t\DD}\ah_n(0)-\tfrac{\ui}{h^2}\int_0^{t} e^{\ui h^{-2}(t-s)\DD}F_n\bigl(\alpha(h^{-2}s)\bigr)\,ds.
\end{align}
We reconstitute $\psi^h(t,x)$ and $\phi^h(t,x)$ from the left-hand side above using \eqref{psi phi def'}.  Employing the notation introduced in \eqref{R def} we may write
$$\psi^h(t) = \mathcal R \bigl[\alpha_n(h^{-2}t)\bigr] \qtq{and}\phi^h(t) = e^{4\ui h^{-2}t} \mathcal R [(-1)^n\alpha_n(h^{-2}t)].$$
Given the hypothesis \eqref{convg}, it thus suffices to prove that
\begin{align}
\lim_{h\to 0} \mathcal R\bigl[\text{RHS}\eqref{ah duh}\bigr]=\text{RHS}\eqref{duhamel 1} \quad \text{in $C([-T,T];L^2_x(\R))$}\label{1}
\end{align}
and 
\begin{align}
&\lim_{h\to 0} e^{4\ui h^{-2}t}\mathcal R\bigl[(-1)^n \text{RHS}\eqref{ah duh}\bigr]=\text{RHS}\eqref{duhamel 2} \quad \text{in $C([-T,T];L^2_x(\R))$}.\label{2}
\end{align}

We first address the convergence of the linear terms on the left-hand sides of \eqref{1} and \eqref{2}.

\begin{lemma}\label{L:1}
Under the hypotheses of Proposition~\ref{P:int eq}, we have
\begin{align*}
&\lim_{h\to 0}\,\Bigl\| \mathcal R\bigl[e^{\ui h^{-2}t\DD}\ah_n(0)\bigr] - e^{\ui t\D}\pe_0   \Bigr\|_{L_t^\infty L_x^2([-T,T]\times\R)}=0,\\
&\lim_{h\to 0}\,\Bigl\|e^{4\ui h^{-2}t}  \mathcal R\bigl[(-1)^n e^{\ui h^{-2}t\DD}\ah_n(0)\bigr]- e^{-\ui t\D}\phi_0   \Bigr\|_{L_t^\infty L_x^2([-T,T]\times\R)}=0.
\end{align*}
\end{lemma}

\begin{proof}
Performing the change of variables $\theta= h\xi$ and using Plancherel and \eqref{data supp}, we estimate
\begin{align*}
&\Bigl\| \mathcal R\bigl[e^{\ui h^{-2}t\DD}\ah_n(0)\bigr] - \ui\tfrac{4\sin^2(\frac\theta2)}{h^2}t\Bigr\} \widehat{\ah}(0,\theta)\,\tfrac{d\theta}{2\pi} - e^{\ui t\D}\pe_0   \Bigr\|_{L_t^\infty L_x^2([-T,T]\times\R)}\\
&\leq \Bigl\| \int\Bigl[\exp\Bigl\{\ui x\xi - \ui\tfrac{4\sin^2(\frac{h\xi}2)}{h^2}t\Bigr\} - \exp\bigl\{ \ui x\xi - \ui t|\xi|^2\bigr\}\Bigr] \widehat{P_{\leq N} \psi_0}(\xi) \,\tfrac{d\xi}{2\pi}\Bigr\|_{L_t^\infty L_x^2([-T,T]\times\R)}\\
&\quad + \Bigl\| e^{\ui t\D}\pe_0 -e^{\ui t\D}P_{\leq N} \psi_0\Bigr\|_{L_t^\infty L_x^2([-T,T]\times\R)}\\
&\lesssim \Big\| \Bigl[\exp\Bigl\{\ui t|\xi|^2 - \ui\tfrac{4\sin^2(\frac{h\xi}2)}{h^2}t\Bigr\} - 1\Bigr]  \widehat{P_{\leq N} \psi_0}(\xi)\Bigr\|_{L_t^\infty L_\xi^2([-T,T]\times\R)} + \bigl\|\widehat{P_{\geq N} \psi_0}\bigr\|_{L_\xi^2}.
\end{align*}
The first claim now follows from the dominated convergence theorem, recalling that $N=h^{-\gamma}$.  The second claim is proved analogously.
\end{proof}

We now turn to the convergence of the nonlinear terms on the left-hand sides of \eqref{1} and \eqref{2}.  We will only present the details for \eqref{1}; the treatment of \eqref{2} is analogous.  To complete the proof of \eqref{duhamel 1}, we must show that
\begin{align}\label{last goal}
\lim_{h\to 0} \mathcal R\Bigl[\tfrac{\ui}{h^2}\int_0^{t} e^{\ui h^{-2}(t-s)\DD}F_n\bigl(\alpha(h^{-2}s)\bigr)\,ds\Bigr] = \mp2\ui\int_0^te^{\ui (t-s)\D}\abs{\pe(s)}^2\pe(s)\,ds
\end{align}
in $C([-T,T];L^2_x(\R))$. Recall that by Lemma~\ref{L:C}, we have $\| \mathcal R\|_{\ell^2_n(\z)\to L_x^2(\R)} \lesssim h^{-\frac12}$.

In view of \eqref{psi phi def}, 
\begin{align}\label{ah id}
\ah_n(h^{-2}t) = h\bigl[\psi^h(t,nh) + (-1)^n e^{-4\ui h^{-2}t}\phi^h(t,nh)\bigr].
\end{align} 

Employing the sharp Fourier cutoff $P_h$ to $|\xi|< h^{-\frac12}$, we define
\begin{align*}
\widetilde\psi^h = P_h\psi^h  \qtq{and} \widetilde\phi^h = P_h\phi^h
\end{align*}
and then
\begin{align}\label{tilde alpha defn}
\widetilde \ah_n(h^{-2}t) = h\bigl[\widetilde\psi^h(t,nh) + (-1)^n e^{-4\ui h^{-2}t}\widetilde\phi^h(t,nh)\bigr].
\end{align}
We will show that as $h\to 0$, the nonlinearity in \eqref{AL} can be replaced by an on-site nonlinearity based solely on this more narrowly Fourier localized sequence.  One reason for introducing this additional localization is to ensure that Lemma~\ref{Lem:norm-relation} may be applied to such nonlinear functions. 

\begin{lemma} \label{L:2} Adopting the notation
\begin{align}\label{sign flip}
\widetilde F_n(t) = 2 h\bigl|\widetilde \alpha_n(h^{-2}t)\bigr|^2 \bigl[\widetilde\psi^h(t,nh) - (-1)^n e^{-4\ui h^{-2}t}\widetilde\phi^h(t,nh)\bigr]
\end{align}
{\upshape(}notice the sign flip relative to \eqref{tilde alpha defn}{\upshape)}, we have
\begin{align*}
\Bigl\| F\bigl(\ah(h^{-2}t)\bigr) - \widetilde F(t) \Bigr\|_{L_t^1\ell_n^2([-T,T]\times\Z)}=o\bigl(h^{\frac52}\bigr) \quad\text{as $h\to 0$.}
\end{align*}
In particular,
$$
\lim_{h\to 0} \, \Bigl\|\mathcal R\Bigl[\tfrac{\ui}{h^2}\int_0^{t} \!\!e^{\ui h^{-2}(t-s)\DD}\Bigl[ F\bigl(\ah(h^{-2}s)\bigr) - \widetilde F(s) \Bigr]\,ds\Bigr]\Bigr\|_{L_t^\infty L_x^2([-T,T]\times\R)}=0.
$$
\end{lemma}

\begin{proof}
In view of \eqref{ah id}, Lemma~\ref{Lem:norm-relation}, and Theorem~\ref{T:precompactness} (specifically \eqref{equi 2}), we may estimate
\begin{align*}
\bigl\|&\ah_n(h^{-2}t) - \widetilde\ah_n(h^{-2}t)\bigr\|_{L_t^\infty \ell_n^2([-T,T]\times\Z)}\\
&\leq h\bigl\| (1-P_h)\psi^h(t,nh)\bigr\|_{L_t^\infty \ell_n^2([-T,T]\times\Z)}+ h\bigl\| (1-P_h)\phi^h(t,nh)\bigr\|_{L_t^\infty \ell_n^2([-T,T]\times\Z)}\\
&\lesssim h^{\frac12}\bigl\| (1-P_h)\psi^h(t)\bigr\|_{L_t^\infty L_x^2([-T,T]\times\R)}+ h^{\frac12}\bigl\| (1-P_h)\phi^h(t)\bigr\|_{L_t^\infty L_x^2([-T,T]\times\R)}\\
&=o\bigl(h^{\frac12}\bigr)  \quad\text{as $h\to 0$.}
\end{align*}
Consequently, by Proposition~\ref{Prop:strichartz-bounds}, we get
\begin{align}\label{part 1}
\Bigl\| F\bigl(\ah(h^{-2}t)\bigr) &- F\bigl(\widetilde\ah(h^{-2}t)\bigr) \Bigr\|_{L_t^1\ell_n^2([-T,T]\times\Z)}\notag\\
&\lesssim \bigl\|\ah(h^{-2}t)\bigr\|_{L_t^4 \ell_n^\infty ([-T, T]\times\Z)}^2 \bigl\|\ah_n(h^{-2}t) - \widetilde\ah_n(h^{-2}t)\bigr\|_{L_t^\infty \ell_n^2 ([-T, T]\times\Z)}\notag\\
&=o\bigl(h^{\frac52}\bigr) \quad\text{as $h\to 0$.}
\end{align}

To continue, we use Lemma~\ref{Lem:norm-relation} and Theorem~\ref{T:precompactness} to estimate
\begin{align*}
\bigl\|&[\widetilde\ah_{n+1}(h^{-2}t)+\widetilde\ah_{n+1}(h^{-2}t)] - 2 h\bigl[\widetilde\psi^h(t,nh) - (-1)^n e^{-4\ui h^{-2}t}\widetilde\phi^h(t,nh)\bigr] \bigr\|_{L_t^\infty \ell_n^2}\\
&\leq 2h\bigl\| \widetilde\psi^h(t,(n+1)h) - \widetilde\psi^h(t,nh)\bigr\|_{L_t^\infty \ell_n^2}
	+ 2h\bigl\| \widetilde\phi^h(t,(n+1)h) - \widetilde\phi^h(t,nh)\bigr\|_{L_t^\infty \ell_n^2}\\
&\leq 2h\bigl\| \psi^h(t,(n+1)h) - \psi^h(t,nh)\bigr\|_{L_t^\infty \ell_n^2}+ 2h\bigl\| \phi^h(t,(n+1)h) - \phi^h(t,nh)\bigr\|_{L_t^\infty \ell_n^2}\\
&\lesssim h^{\frac12}\bigl\| \psi^h(t,\cdot +h) - \psi^h(t)\bigr\|_{L_t^\infty L_x^2}+ h^{\frac12}\bigl\| \phi^h(t,\cdot +h) - \phi^h(t)\bigr\|_{L_t^\infty L_x^2}
=o\bigl(h^{\frac12}\bigr) 
\end{align*}
as $h\to 0$.  Consequently, using Proposition~\ref{Prop:strichartz-bounds} as in \eqref{part 1}, we obtain
\begin{align}\label{part 2}
\Bigl\| F\bigl(\widetilde\ah(h^{-2}t)\bigr)- \widetilde F(t) \Bigr\|_{L_t^1\ell_n^2([-T,T]\times\Z)}=o\bigl(h^{\frac52}\bigr) \quad\text{as $h\to 0$.}
\end{align}
Combining \eqref{part 1} and \eqref{part 2} settles the lemma.
\end{proof}

Expanding out the definition, we have 
\begin{align*}
 \widetilde F_n(t)& = 2h^3\bigl|\widetilde\psi^h(t,nh)\bigr|^2\widetilde\psi^h(t,nh) -2h^3e^{-8\ui h^{-2}t}(\widetilde\phi^h(t,nh))^2\overline{\widetilde\psi^h(t,nh)}\\
 &\quad - 2 (-1)^n h^3 e^{-4\ui h^{-2}t} \Bigl[ \bigl|\widetilde\phi^h(t,nh)\bigr|^2\widetilde\phi^h(t,nh) - e^{8\ui h^{-2}t}(\widetilde\psi^h(t,nh))^2\overline{\widetilde\phi^h(t,nh)}\,\Bigr].
\end{align*}
By looking at the Fourier supports, we see that only the top row of terms contributes to the left-hand side of \eqref{1} and the second row contributes only to that of \eqref{2}.

Our next result shows how the temporal non-resonance of the unexpected terms in the expansion for $\widetilde F_n$ (namely those involving both $\widetilde\psi^h$ and $\widetilde\phi^h$) cause them to drop out in the limit $h\to 0$.  

\begin{lemma}\label{L:3} Let $\mathcal E_n$ denote a cubic polynomial in $\widetilde\psi^h(nh)$, $\widetilde\phi^h(nh)$, and their complex conjugates and let $m$ be a non-zero integer.  Then
\begin{align*}
\Bigl\| \int_0^{t} e^{\ui h^{-2}(t-s)\DD}e^{\ui m h^{-2}s}\mathcal E_n(s)\,ds \Bigr\|_{L_t^\infty\ell^2_n([-T,T]\times\Z)}=o\bigl(h^{-\frac12}\bigr) \quad\text{as $h\to 0$.}
\end{align*}
In particular,
$$
\lim_{h\to 0} \Bigl\|\mathcal R\Bigl[ \tfrac{\ui}{h^2}\int_0^{t} e^{\ui h^{-2}(t-s)\DD}e^{\ui m h^{-2}s} h^3 \mathcal E_n(s)\,ds\Bigr] \Bigr\|_{L_t^\infty L^2_x([-T,T]\times\R)}=0.
$$
\end{lemma}

\begin{proof}
We decompose
\begin{align*}
&\int_0^{t} e^{\ui h^{-2}(t-s)\DD}e^{\ui m h^{-2}s}\mathcal E_n(s)\,ds \\
&= \tfrac12\int_0^{t} e^{\ui h^{-2}(t-s)\DD}e^{\ui m h^{-2}s}\mathcal E_n(s)\,ds- \tfrac12\int_0^{t} e^{\ui h^{-2}(t-s)\DD}e^{\ui m h^{-2}(s+\frac{\pi h^2} m)}\mathcal E_n(s)\,ds\\
& =  \tfrac12 \int_0^{t} e^{\ui m h^{-2}s+\ui h^{-2}(t-s)\DD}\bigl[1-e^{\ui \frac{\pi} m\DD}\bigr]\mathcal E_n(s)\,ds \\
&\quad + \tfrac12 \int_0^{t}  e^{\ui h^{-2}(t+\frac{\pi h^2} m-s)\DD}e^{\ui m h^{-2}s}\bigl[\mathcal E_n(s)- \mathcal E_n\bigl(s-\tfrac{\pi h^2} m\bigr)\bigr]\,ds\\
&\quad +\tfrac12\int_0^{\frac{\pi h^2}m} e^{\ui h^{-2}(t+\frac{\pi h^2} m-s)\DD}e^{\ui m h^{-2}s}\mathcal E_n\bigl(s-\tfrac{\pi h^2} m\bigr)\,ds\\
&\quad -\tfrac12\int_{t}^{t+\frac {\pi h^2} m} e^{\ui h^{-2}(t+\frac{\pi h^2} m-s)\DD}e^{\ui m h^{-2}s}\mathcal E_n\bigl(s-\tfrac{\pi h^2} m\bigr)\,ds.
\end{align*}

Using the discrete Strichartz inequality Proposition~\ref{P:discrete Strichartz} followed by Lemma~\ref{Lem:norm-relation} and Proposition~\ref{Prop:strichartz-bounds}, we estimate the contribution of the last two terms in our decomposition by 
\begin{align*} 
\bigl(\tfrac{\pi h^2} m\bigr)^\frac12 \|\mathcal E_n\|_{L_t^2 \ell_n^2([-T,T]\times\Z)}& \lesssim h\bigl[\|\widetilde\psi^h(nh)\|_{L_t^6\ell_n^6([-T,T]\times\Z)}^3+\|\widetilde\phi^h(nh)\|_{L_t^6\ell_n^6([-T,T]\times\Z)}^3\bigr]\\
&\lesssim h^{\frac12}\bigl[\|\widetilde\psi^h\|_{L_{t,x}^6([-T,T]\times\R)}^3+\|\widetilde\phi^h\|_{L_{t,x}^6([-T,T]\times\R)}^3\bigr]\lesssim h^\frac12,
\end{align*}
which is acceptable.

In view of the Fourier localization imposed on the functions $\widetilde\psi^h$ and $\widetilde\phi^h$, we have $\supp(\widehat{\mathcal E_n})\subset [-3h^{\frac12}, 3h^{\frac12}]$.  Thus, estimating as above, we may bound the contribution of the first term in our decomposition by 
\begin{align*}
\Bigl\|\int_0^{t} &e^{\ui m h^{-2}s+\ui h^{-2}(t-s)\DD}\bigl[1-e^{\ui \frac{\pi} m\DD}\bigr]\mathcal E_n(s)\,ds\Bigr\|_{L_t^\infty\ell^2_n([-T,T]\times\Z)}\\
&\qquad\lesssim \bigl\| \bigl[1-e^{\ui \frac{\pi} m\DD}\bigr]\mathcal E_n \bigr\|_{L_t^1\ell^2_n([-T,T]\times\Z)}\lesssim h \|\mathcal E_n \bigr\|_{L_t^1\ell^2_n([-T,T]\times\Z)}\lesssim h^{\frac12},
\end{align*}
which is acceptable.

This leaves us to estimate the second term in our decomposition, which represents the central question at hand.  The key idea here is to exploit the equicontinuity in time of the functions $\widetilde\psi^h$ and $\widetilde\phi^h$ proved in Theorem~\ref{T:precompactness}.  As $\supp(\widehat{\mathcal E_n})\subset [-3h^{\frac12}, 3h^{\frac12}]$, we may employ the frequency-localized Strichartz estimates from Proposition~\ref{P:loc Strichartz} and scaling to estimate
\begin{align*}
&\Bigl\|\int_0^{t}  e^{\ui h^{-2}(t+\frac{\pi h^2} m-s)\DD}e^{\ui m h^{-2}s}\bigl[\mathcal E_n(s)- \mathcal E_n\bigl(s-\tfrac{\pi h^2}m\bigr)\bigr]\,ds\Bigr\|_{L_t^\infty\ell^2_n}\\
&\lesssim h^\frac13 \bigl\| \mathcal E_n(s)- \mathcal E_n\bigl(s-\tfrac{\pi h^2} m\bigr)\bigr\|_{L^{\frac65}_t \ell_n^{\frac65}([-T,T]\times\Z)}\\
&\lesssim  T^\frac12 h^{-\frac12} \Bigl[\|\psi^h\|_{L_{t,x}^6([-T,T]\times\R)}^2 +\|\phi^h\|_{L_{t,x}^6([-T,T]\times\R)}^2\Bigr]\\
&\qquad\times\Bigl[\bigl\|\psi^h -\psi^h\bigl(\cdot -\tfrac{\pi h^2}m\bigr)\bigr\|_{L_t^\infty L_x^2([-T,T]\times\R)} +\bigl\|\psi^h -\psi^h\bigl(\cdot -\tfrac{\pi h^2}m\bigr)\bigr\|_{L_t^\infty L_x^2([-T,T]\times\R)}\Bigr]\\
&=o\bigl(h^{-\frac12}\bigr) \quad\text{as $h\to 0$}
\end{align*}
where we used Lemma~\ref{Lem:norm-relation}, Proposition~\ref{Prop:strichartz-bounds}, and Theorem~\ref{T:precompactness} in the last two lines.

This completes the proof of the lemma.
\end{proof}

We finally consider the main contribution to the left-hand side of \eqref{last goal}.

\begin{lemma}\label{L:4} We have 
\begin{align*}
\lim_{h\to 0} \mathcal R\Bigl[\tfrac{\ui}{h^2}\int_0^{t} e^{\ui h^{-2}(t-s)\DD}2h^3\bigl(|\widetilde\psi^h|^2\widetilde\psi^h\bigr)(s,nh)\,ds\Bigr] = \mp2\ui\int_0^te^{\ui (t-s)\D}\abs{\pe(s)}^2\pe(s)\,ds
\end{align*}
in $C([-T,T];L^2_x(\R))$.
\end{lemma}

\begin{proof}
Combining hypothesis \eqref{convg} and Lemmas~\ref{L:1} through \ref{L:3}, we know that the term on the left-hand side above converges in $C([-T,T];L^2_x(\R))$.  Thus, it suffices  to identify its limit via duality.  Let $f\in \mathcal S(\R)$ be such that $\widehat f \in C_c^\infty(\R)$.  Note that for $h$ sufficiently small we have $\supp \widehat f \subseteq \{\xi: \, |h\xi|<\frac\pi 2\}$.  Thus, we may use Plancherel to compute
\begin{align*}
\bigl\langle \mathcal R &\Bigl[\tfrac{\ui}{h^2} \int_0^{t} e^{\ui h^{-2}(t-s)\DD}2h^3\bigl(|\widetilde\psi^h|^2\widetilde\psi^h\bigr)(s,nh)\,ds\Bigr] ,f\bigr\rangle\\
&=2h\int_0^{t} \bigl\langle  \widehat{|\widetilde\psi^h|^2\widetilde\psi^h}(s,h\xi), \exp\bigl\{\ui\tfrac{4\sin^2(\frac{h\xi}2)}{h^2}(t-s)\bigr\}
	\widehat{f} (\xi)\bigr\rangle\, ds\\
&=2h\int_0^{t} \sum_n \bigl(|\widetilde\psi^h|^2\widetilde\psi^h\bigr)(s,nh) \overline{\bigl[e^{-i(t-s)\Delta} f\bigr]}(nh)\, ds\\
&\quad + 2h\int_0^{t}\sum_n  \bigl(|\widetilde\psi^h|^2\widetilde\psi^h\bigr)(s,nh) \!\int e^{inh\xi}\Bigl[e^{\ui\tfrac{4\sin^2(\frac{h\xi}2)}{h^2}(t-s)}-e^{\ui|\xi|^2(t-s)}\Bigr]
	\widehat f (\xi)\, d\xi\, ds.
\end{align*}
Using Plancherel, Lemma~\ref{Lem:norm-relation}, and Proposition~\ref{Prop:strichartz-bounds} together with the dominated convergence theorem, we may bound the second summand above by
\begin{align*}
&h\int_0^t\bigl\| \bigl(|\widetilde\psi^h|^2\widetilde\psi^h\bigr)(s,nh) \bigr\|_{\ell_n^2} h^{-\frac12} \Bigl\| \Bigr[e^{\ui[ \xi^2 -\tfrac{4\sin^2(\frac{h\xi}2)}{h^2}](t-s)}-1\Bigr]\widehat f\Bigr\|_{L_\xi^2}\,ds\\
&\qquad\lesssim T^\frac12 \|\widetilde\psi^h\|_{L_{t,x}^6([-T, T]\times\R)}^3 \, o(1) = o(1) \quad\text{as $h\to 0$}.
\end{align*}

Noting that the Fourier support of $|\widetilde\psi^h|^2\widetilde\psi^h$ is contained in $[-3h^{-\frac12}, 3h^{-\frac12}]$ and invoking Lemma~\ref{L:sums to int}, we see that the first summand equals
\begin{align*}
2\int_0^{t} \int \bigl(|\widetilde\psi^h|^2\widetilde\psi^h\bigr)(s,x) \overline{\bigl[e^{-i(t-s)\Delta} f\bigr]}(x)\, dx\, ds =2\int_0^{t} \bigl\langle e^{i(t-s)\Delta}\bigl(|\widetilde\psi^h|^2\widetilde\psi^h\bigr)(s), f\bigr\rangle \, ds.
\end{align*}
That this converges to the desired limit as $h\to 0$ follows readily from the Strichartz inequality Proposition~\ref{P:Strichartz},  hypothesis \eqref{convg}, Theorem~\ref{T:precompactness}, Proposition~\ref{Prop:strichartz-bounds} and \eqref{6 bdd}:
\begin{align*}
\Bigl\|  \int_0^{t} \bigl\langle e^{i(t-s)\Delta}&\bigl[|\widetilde\psi^h|^2\widetilde\psi^h - |\psi|^2\psi\bigr](s), f\bigr\rangle \, ds \Bigr\|_{L_t^\infty ([-T,T])}\\
&\lesssim T^{\frac12}\|\widetilde\psi^h - \psi\|_{L_t^\infty L_x^2} \bigl[\|\widetilde\psi^h\|_{L_{t,x}^6}^2 +\|\psi\|_{L_{t,x}^6}^2\bigr] \|f\|_{L_x^2}\\
&\lesssim_T \|\psi^h - \psi\|_{L_t^\infty L_x^2}+ \|P_{|\xi|\geq h^{-\frac12}}\psi^h \|_{L_t^\infty L_x^2} = o(1)  \quad\text{as $h\to 0$},
\end{align*}
where all spacetime norms are taken over $[-T,T]\times\R$.
\end{proof}

The proof of Proposition~\ref{P:int eq} is now complete.  Indeed, Lemmas~\ref{L:1} though \ref{L:4} show that $\psi$ satisfies \eqref{duhamel 1}.  The proof that $\phi$ satisfies \eqref{duhamel 2} follows from parallel arguments. \end{proof}

We are finally ready to prove our main result:

\begin{proof}[Proof of Theorem~\ref{T:main}]  As noted at the beginning of this section, Theorem~~\ref{T:precompactness} guarantees that every sequence $h\rt 0$ admits a subsequence so that both $\pe^{h}$ and $\phi^{h}$ converge in $C([-T,T];L^2_x(\r))$.   By Proposition~\ref{P:int eq}, the limiting functions lie in $L^6_{t,x}$ and solve the integral equations \eqref{duhamel 1} and \eqref{duhamel 2}.

These integral equations admit only one solution in $C_t L^2_x \cap L^6_{t,x}$ as is easily shown by contraction mapping using the estimates recalled in Proposition~\ref{P:Strichartz}.   Originating in \cite{MR915266}, this is now the textbook approach to the construction of solutions to \eqref{NLS system}.

As all subsequential limits agree, it follows that the original sequences $\pe^{h}$ and $\phi^{h}$ converge as $h\to0$ without passing to subsequences at all.  Moreover, as noted above, the resulting limits are the unique solutions to the evolutions \eqref{NLS system} with initial data $\psi_0,\phi_0$.
\end{proof}

\end{document}